\documentclass[a4paper,12pt]{article}

\usepackage{amsmath}
\usepackage{amsthm}
\usepackage{enumerate}
\usepackage{amssymb,amsfonts,latexsym,mathtools}
\usepackage{mathrsfs}
\usepackage{bm}
\usepackage{cases}
\usepackage{color}

\newcommand{\Levy}{L\'{e}vy}

\newcommand{\deq}{\stackrel{d}{=}}

\newcommand{\R}{\mathbb{R}}
\newcommand{\F}{\mathscr{F}}

\newcommand{\Q}{\mathbb{Q}}
\newcommand{\e}{\varepsilon}

\newcommand{\cadlag}{c\`{a}dl\`{a}g}
\renewcommand{\P}{\mathbb{P}}
\usepackage{multirow}

\numberwithin{equation}{section}

\makeatletter
\renewcommand\section{\@startsection {section}{1}{\z@}%
{-3.5ex \@plus -1ex \@minus -.2ex}%
{2.3ex \@plus.2ex}%
{\normalfont\large\bf}}
\makeatother

\makeatletter
\renewcommand\subsection{\@startsection {subsection}{1}{\z@}%
{-3.5ex \@plus -1ex \@minus -.2ex}%
{2.3ex \@plus.2ex}%
{\normalfont\normalsize\bf}}
\makeatother

\theoremstyle{plain}
\newtheorem{thm}{Theorem}[section]
\newtheorem{lem}[thm]{Lemma}
\newtheorem{cor}[thm]{Corollary}
\newtheorem{prop}[thm]{Proposition}
\theoremstyle{definition}
\newtheorem{Rem}[thm]{Remark}
\allowdisplaybreaks

\usepackage[dvipdfmx]{hyperref}
\hypersetup{
setpagesize=false,
 bookmarksnumbered=true,
 bookmarksopen=true,
 colorlinks=true,
 linkcolor=blue,
 citecolor=red,
}

\begin{document}
\begin{center}
\Large \textbf{Multi-point local time penalizations with various clocks for one-dimensional \Levy\ processes}
\end{center}
\begin{center}
Kohki IBA\footnote{Graduate School of Science, The University of Osaka, Japan. E-mail: kohki.iba [at] gmail.com}
\end{center}
\begin{abstract}
We study the penalization problem with various clocks where the weight is given as the exponential functional of multi-point local times for one-dimensional \Levy\ processes. The limit processes may vary according to the choice of random clock, and are singular to the original process.
\end{abstract}


\section{Introduction}
A penalization problem is to study the long-time limit of the form
\begin{align}
\label{penal}
  \lim_{\tau\to \infty}\frac{\P_x[F_s\cdot \Gamma_\tau]}{\P_x[\Gamma_\tau]},
\end{align}
for all bounded $\F_s$-adapted functional $F_s$, where $((X_t)_{t\ge 0},(\F_t)_{t\ge 0},(\P_x)_{x\in \R})$ is a Markov process, $\P_x[\cdot]$ denotes the expectation with respect to the measure $\P_x$, $(\Gamma_t)_{t\ge 0}$ is a non-negative process called a \emph{weight}, and $\tau$ is a net of parametrized random times tending to infinity called a \emph{clock}. Here, as the random clock $\tau$, we adopt one of the following:
\begin{itemize}
  \item (C) \emph{Constant clock}: $\tau=t$ as $t\to \infty.$
  \item (Ex) \emph{Exponential clock}: $\tau=(\bm{e}_q)$ as $q\to 0+$, where $\bm{e}_q$ has the exponential distribution with parameter $q>0$ and is independent of the process $(X_t)_{t\ge 0}.$
  \item (OH) \emph{One-point hitting time clock}: $\tau=(T_b)$ as $b\to \pm \infty$, where $T_b$ is the first hitting time at point $b\in \R$.
  \item (TH) \emph{Two-point hitting time clock}: $\tau=(T_c\wedge T_{-d})$ as $c,d\to \infty$ and $\frac{d-c}{c+d}\to \gamma \in [-1,1]$, which denote $(c,d)\stackrel{(\gamma)}{\to}\infty.$
  \item (IL) \emph{Inverse local time clock}: $\tau=(\eta_u^b)$ as $b\to \pm \infty$, where $(\eta_u^b)_{u\ge 0}$ is the inverse local time at $b\in \R.$
\end{itemize}

To solve this problem, we want to find a function $\rho(\tau)$ of the clock $\tau$ and a $(\F_s)$-martingale $(M_s^\Gamma)_{s\ge 0}$ such that
\begin{align}
  \lim_{\tau\to \infty}\rho(\tau)\P_{x}\left[F_s\cdot \Gamma_\tau\right]=\P_x[F_s\cdot M_s^\Gamma]
\end{align}
hold. If $M_0^\Gamma>0$ under $\P_x$, this convergence implies
\begin{align}
  \lim_{\tau\to \infty}\frac{\P_x[F_s\cdot \Gamma_\tau]}{\P_x[\Gamma_\tau]}=\P_x\left[F_s\cdot \frac{M_s^\Gamma}{M_0^\Gamma}\right],
\end{align}
which solves the penalization problem (\ref{penal}) (see, Sections \ref{Sec3}, \ref{Sec4}, \ref{Sec5},\ and \ref{Sec6} for the details).

\subsection{Backgrounds of the penalization problems}
This problem has been studied for various one-dimensional processes. Here, we introduce previous researches which the weight is expressed as a function of the local time.
\begin{center}
\begin{tabular}{|c|c|c|c|} \hline
 Process $X$ & Paper & Weight $\Gamma$ & Clock $\tau$ \\ \hline\hline

  \multirow{2}{*}{Brownian motion} &  Roynette et al. \cite{RVY1}
& $L_t^0$ & \multirow{2}{*}{(C)}\\ \cline{2-3}
  & Roynette et al. \cite{RVY2}& $e^{-\int L_t^x q(dx)}$ & \\ \hline \hline
 
 Bessel process &  Roynette et al. \cite{RVY-B} & $f(L_t^0),\ f(L_t^0)e^{\lambda X_t}$ & (C)\\ \hline \hline
 
 Random walk & Debs \cite{D} & $f(L_n^0)$ & (C) \\ \hline \hline
 
\multirow{5}{*}{Diffusion} &  Salminen-Vallois \cite{SV}
& $f(L_t^0)$ & \multirow{2}{*}{(C)}\\ \cline{2-3}
  & Profeta \cite{P2} & $e^{\pm \lambda L_t^0}$ & \\ \cline{2-4}
  
  &Profeta et al. \cite{PYY} & $f(L_t^0)$ & \begin{tabular}{c}
   (Ex)\\
   (OH)\\
   (IL)
 \end{tabular} \\ \hline \hline
 
Stable process & Yano et al. \cite{YYY} & $f(L_t^0),\ e^{-\int L_t^x q(dx)}$ & (C) \\ \hline \hline
 
\multirow{2}{*}{\Levy \ process} & Takeda-Yano \cite{TY}
& $f(L_t^0)$ & \multirow{2}{*}{
  \begin{tabular}{c}
    (Ex),(OH)\\
    (TH),(IL)
  \end{tabular}}\\ \cline{2-3}
  & Iba-Yano \cite{IY-3} & $e^{-\lambda L_t^a-\lambda' L_t^b}$
  & \\  \hline
\end{tabular}
\end{center}
Here, $f$ is a given function satisfying appropriate conditions, $(L_t^a)_{t\ge 0}$ represents the local time of point $a\in \R$, $\lambda$ and $\lambda'$ are positive constants, and $q(dx)$ is a Radon measure satisfying appropriate conditions. In particular, the case of the weight function $f(L_t^0)$ is referred to as the \emph{local time penalization}, and the case of the weight function $e^{-\int L_t^x q(dx)}$ is referred to as the \emph{Kac killing penalization}. 

In the local time penalization problem with $f(x)=e^{-\lambda x}$, the weight function decreases each time the process hits $0$, as the local time increases upon each visit to $0$. This can be interpreted as the process being penalized each time it hits $0$. The limiting measure obtained in the penalization problem is called the \emph{penalized measure}, and it is singular to the original measure $\P_x$ (see, Section \ref{Sec8} for the details).

A special case of the penalization problem with weight $\Gamma_s=1_{\{L_s^0=0\}}=1_{\{s<T_0\}}$ is called the problem of \emph{conditioning to avoid zero}. The Brownian conditioning problem using random clocks date back to the \emph{taboo\ process} of Knight \cite{Kn}. This problem has been generalized to one-dimensional \Levy\ processes and various sets (see, e.g., \cite{Iba} for the details).

\subsection{Main results}
We consider a one-dimensional \Levy\ process $(X_t)_{t\ge 0}$ which is recurrent. For the characteristic exponent $\Psi$ of $X$, we always assume the condition
\begin{align*}
  \int_0^\infty \left|\frac{1}{q+\Psi(\lambda)}\right|d\lambda <\infty\qquad \text{for}\ q>0.
  \tag{\bf{A}}
\end{align*}
Let $(L_t^a)_{t\ge 0}$ denote the local time process of $a\in \R$ for $(X_t)_{t\ge 0}$ (subject to a suitable normalization). The weight process we consider is given as
\begin{align}
  \Gamma_t^{(n)}:=\exp\left(-\sum_{k=1}^n \lambda_{a_k}L_t^{a_k}\right)
\end{align}
for $n\ge 2$, distinct points $a_1,...,a_n\in \R$, and positive constants $\lambda_{a_1},...,\lambda_{a_n}>0.$ This weight process can be considered to be the Kac killing penalization with
\begin{align}
  q(dx)=\lambda_{a_1}\delta_{a_1}(dx)+\cdots +\lambda_{a_n}\delta_{a_n}(dx).
\end{align}
The case $\lambda_{a_1}=\cdots =\lambda_{a_n}=\infty$ is formally considered a conditioning to avoid $n$-points. This case are discussed in Iba \cite{Iba}.

Our main theorem is as follows:\\

\begin{thm}
\label{Thm1}
The following penalization limits hold:
\begin{align}
  \text{(Ex)}&\qquad \lim_{q\to 0+}\frac{\P_x[F_s\cdot \Gamma_{\bm{e}_q}^{(n)}]}{\P_x[\Gamma_{\bm{e}_q}^{(n)}]}=\P_x\left[F_s\cdot \frac{\varphi_{A_n}^{\lambda_{a_1},...,\lambda_{a_k}}(X_s)\Gamma_s^{(n)}}{\varphi_{A_n}^{\lambda_{a_1},...,\lambda_{a_k}}(x)}\right],\\
  \text{(OH)}&\qquad  \lim_{b\to \pm \infty}\frac{\P_x[F_s\cdot \Gamma_{T_b}^{(n)}]}{\P_x[\Gamma_{T_b}^{(n)}]}=\P_x\left[F_s\cdot \frac{\varphi_{A_n}^{(\pm 1),\lambda_{a_1},...,\lambda_{a_n}}(X_s)\Gamma_s^{(n)}}{\varphi_{A_n}^{(\pm 1),\lambda_{a_1},...,\lambda_{a_n}}(x)}\right],\\
  \text{(TH)}&\qquad \lim_{(c,d)\stackrel{(\gamma)}{\to}\infty}\frac{\P_x[F_s\cdot \Gamma_{T_{c}\wedge T_{-d}}^{(n)}]}{\P_x[\Gamma_{T_{c}\wedge T_{-d}}^{(n)}]}=\P_x\left[F_s\cdot \frac{\varphi_{A_n}^{(\gamma),\lambda_{a_1},...,\lambda_{a_n}}(X_s)\Gamma_s^{(n)}}{\varphi_{A_n}^{(\gamma),\lambda_{a_1},...,\lambda_{a_n}}(x)}\right],\\
  \text{(IL)}&\qquad \lim_{b\to \pm \infty}\frac{\P_x[F_s\cdot \Gamma_{\eta_u^b}^{(n)}]}{\P_x[\Gamma_{\eta_u^b}^{(n)}]}=\P_x\left[F_s\cdot \frac{\varphi_{A_n}^{(\pm 1),\lambda_{a_1},...,\lambda_{a_n}}(X_s)\Gamma_s^{(n)}}{\varphi_{A_n}^{(\pm 1),\lambda_{a_1},...,\lambda_{a_n}}(x)}\right].
\end{align}
For $-1\le \gamma\le 1$, the functions $\varphi_{A_n}^{(\gamma),\lambda_{a_1},...,\lambda_{a_n}}(x)$ are defined by
\begin{align}
\label{func-1}
  \varphi_{A_n}^{(\gamma),\lambda_{a_1},...,\lambda_{a_n}}(x):=\varphi_{A_n}^{(\gamma)}(x)+\left\langle\bm{p}_n(x),(\mathbb{E}_n-\mathbb{J}_n)^{-1}\bm{i}_n^{(\gamma)}\right\rangle,
\end{align}
where $\langle\cdot,\cdot\rangle$ denotes the standard inner product on $\R^n$, 
\begin{align}
  \varphi_{A_n}^{(\gamma)}(x):=h^{(\gamma)}(x-a_n)-\sum_{k=1}^{n-1}h^{(\gamma)}(a_k-a_n)\P_x(T_{a_k}=T_{A_n}),
\end{align}
which is defined by (1.10) of Iba \cite{Iba},
\begin{align}
  \bm{p}_n(x)&:=\begin{pmatrix}
    \P_x(T_{a_1}=T_{A_n})\\
    \P_x(T_{a_2}=T_{A_n})\\
    \vdots\\
    \P_x(T_{a_n}=T_{A_n})
  \end{pmatrix},\ \bm{i}_n^{(\gamma)}:=\begin{pmatrix}
    I_1^{(\gamma)}\\
    I_2^{(\gamma)}\\
    \vdots\\
    I_n^{(\gamma)}
  \end{pmatrix},\\
  I_{k}^{(\gamma)}&:=\frac{1}{\lambda_{a_k}}-\sum_{\substack{i;\ i\le n\\ i\neq k}}J_{k,i}\left(\frac{1}{\lambda_{a_k}}+h^{(\gamma)}(a_i-a_k)\right),\\
  \mathbb{J}_n&:=(J_{k,i})_{k,i=1}^n,\  J_{k,i}:=\begin{cases}
  \P_{a_k}\left[e^{-\lambda_{a_k}L_{T_{a_i}}^{a_k}},\ T_{a_i}=T_{A_n\setminus \{a_k\}}\right]&\text{for}\ k\neq i,\\
  0&\text{for}\ k=i,
  \end{cases}
\end{align}
and $\mathbb{E}_n$ is the $n\times n$ identity matrix. The functions $h$ and $h^{(\gamma)}$ are defined in (\ref{h}) or (\ref{hgam}), respectively.
\end{thm}

We will state Theorem \ref{Thm1} in a more general form with a proof as Theorems \ref{penal-1}, \ref{penal-2}, \ref{penal-3},\ and \ref{penal-4}.\\

\begin{Rem}
When $\gamma=0$, there is another expression of $I_k^{(0)}$:
\begin{align}
\label{Ieq}
  I_k^{(0)}=\frac{1}{\bm{n}^{a_k}(T_{A_n\setminus \{a_k\}}<\infty)+\lambda_{a_k}}\left(1-\sum_{\substack{i;\ i\le n\\ i\neq k}}h(a_i-a_k)\bm{n}^{a_k}(T_{a_i}=T_{A_n\setminus \{a_k\}}<\infty)\right),
\end{align}
where $\bm{n}^{a_k}$ is the excursion measure away from $a_k$ with a suitable normalization. We will prove this identity in Section \ref{S9}.
\end{Rem}

\subsection{Organization}
This paper is organized as follows. In Section \ref{Sec2}, we prepare some general results of one-dimensional \Levy\ processes. In Sections \ref{Sec3}, \ref{Sec4}, \ref{Sec5},\ and \ref{Sec6}, we discuss the penalization results with the exponential clock, the hitting time clock, the two-point hitting time clock, and the inverse local time clock, respectively. The structure of Sections \ref{Sec3}, \ref{Sec4}, \ref{Sec5},\ and \ref{Sec6} are similar: we first compute the expectation $\P_x[\Gamma_\tau]$, then we study the limit of $\rho(\tau)\P_x[\Gamma_\tau]$ as $\tau\to \infty$, and finally state and prove the main penalization results. In section \ref{Sec7}, we consider the spacial case of $n=2$. In Section \ref{Sec8}, we study the $n$-point penalized measure and its properties. In Section \ref{S9}, we prove the equation (\ref{Ieq})


\section{Preliminaries}
\label{Sec2}
\subsection{\Levy\ process and resolvent density}
Let $((X_t)_{t\ge 0},(\P_x)_{x\in \R})$ be the canonical representation of a one-dimensional \Levy\ process with starting from $x\in \R$, $\P_x$-a.s. For $t\ge 0$, we denote by $\F_t^X:=\sigma(X_s,\ 0\le s\le t)$ the natural filtration of $(X_t)_{t\ge 0}$, and write $\F_t:=\bigcap_{\e>0}\F_{t+\e}^X$. For a set $A\subset \R$, let $T_A$ be the first hitting time of $A$ for $(X_t)_{t\ge 0}$, i.e.,
\begin{align}
  T_{A}:=\inf \{t\ge 0;\ X_t\in A\}.
\end{align}
For simplicity, we denote $T_{\{a\}}$ as $T_a$ for $a\in \R.$ For $\lambda\in \R$, we denote by $\Psi(\lambda)$ the characteristic exponent of $(X_t)_{t\ge 0}$, i.e., $\Psi(\lambda)$ satisfies
\begin{align}
  \P_0\left[e^{i\lambda X_t}\right]=e^{-t\Psi(\lambda)}\qquad \text{for}\ t\ge 0.
\end{align}

Throughout this paper, we always assume $((X_t)_{t\ge 0},\P_0)$ is recurrent, and always assume the condition
\begin{align*}
\label{A}
  \int_0^\infty \left|\frac{1}{q+\Psi(\lambda)}\right|d\lambda<\infty\qquad \text{for}\ q>0. \tag{\bf{A}}
\end{align*}
Under this condition, it is known that $(X_t)_{t\ge 0}$ has a bounded continuous resolvent density $r_q(x)$. It is satisfies
\begin{align}
  \int_{\R}f(x)r_q(x)dx=\P_x\left[\int_0^\infty e^{-qt}f(X_t)dt\right]\qquad \text{for}\ q>0\ \text{and}\ f\ge 0
\end{align}
(see, e.g., Theorems II.16 and II.19 of \cite{Ber}). Moreover, it is known that the formula between the first hitting time of $a\in \R$ and the resolvent density $r_q(x)$:
\begin{align}
\label{A-1}
  \P_x\left[e^{-qT_a}\right]=\frac{r_q(a-x)}{r_q(0)}\qquad \text{for}\ q>0\ \text{and}\ x\in \R
\end{align}
(see, e.g., Corollary II.18 of \cite{Ber}).

\subsection{Local time and excursion}
For $a\in \R$, we denote by $\mathscr{D}^a$ the set of \cadlag\ paths $e:[0,\infty)\to \R\cup \{\Delta\}$ such that
\begin{align}
  \begin{cases}
    e(t)\in \R\setminus \{a\}\qquad &\text{for}\ 0<t<\zeta(e),\\
    e(t)=\Delta\qquad &\text{for}\ t\ge \zeta(e),
  \end{cases}
\end{align}
where the point $\Delta\notin \R$ is a cemetery state and $\zeta$ is the excursion length, i.e.,
\begin{align}
  \zeta=\zeta(e):=\inf \{t>0;\ e(t)=\Delta\}.
\end{align}
Let $\Sigma^a$ denote the $\sigma$-algebra on $\mathscr{D}^a$ generated by cylinder sets. 

Thanks to condition (\ref{A}), we can define a local time of $a\in \R$, which we denote by $(L_t^a)_{t\ge 0}$. It is known that $(L_t^a)_{t\ge 0}$ is continuous in $t$ and satisfies
\begin{align}
  \P_x\left[\int_0^\infty e^{-qt}dL_t^a\right]=r_q(a-x)\qquad \text{for}\ q>0\ \text{and}\ x\in \R
\end{align}
(see, e.g., Section V of \cite{Ber}). 

Let $(\eta_l^a)_{l\ge 0}$ be an inverse local time, i.e.,
\begin{align}
  \eta_l^a:=\inf \{t>0;\ L_t^a>l\}.
\end{align}
It is known that the process $((\eta_l^a)_{l\ge 0},\P_a)$ is a subordinator which has the Laplace exponent
\begin{align}
  \P_a\left[e^{-q\eta_l^a}\right]=e^{-\frac{l}{r_q(0)}}\qquad \text{for}\ l\ge 0\ \text{and}\ q>0
\end{align}
(see, e.g., Proposition V.4 of \cite{Ber}).

We denote $(\epsilon_l^a(t))_{t\ge 0}$ for an excursion away from $a\in \R$ which starts at local time $l\ge 0$, i.e.,
\begin{align}
  \epsilon_l^a(t):=\begin{cases}
    X_{t+\eta_{l-}^a}\qquad &\text{for}\ 0\le t<\eta_l^a-\eta_{l-}^a,\\
    \Delta &\text{for}\ t\ge \eta_l^a-\eta_{l-}^a.
  \end{cases}
\end{align}
Then, $(\epsilon_l^a)_{l\ge 0}$ is a Poisson point process, and we write $\bm{n}^a$ for the characteristic measure of $(\epsilon_l^a)_{l\ge 0}$. It is known that $(\mathscr{D}^a,\Sigma^a,\bm{n}^a)$ is a $\sigma$-finite measure space (see, e.g., Section IV of \cite{Ber}), and there is the formula 
\begin{align}
\label{A-2}
  \bm{n}^a\left[1-e^{-qT_a}\right]=\frac{1}{r_q(0)}
\end{align}
(see, e.g., Equation (2.4) of \cite{TY}).

We define
\begin{align}
  \bm{N}^a(B):=\#\{(l,e)\in B;\ \epsilon_l^a=e\}\qquad \text{for}\ B\in \mathscr{B}(0,\infty)\otimes \Sigma^a.
\end{align}
Then, $\bm{N}^a$ is a Poisson random measure with its intensity measure $ds\otimes \bm{n}^a(de)$.

The excursion measure $\bm{n}^a$ has the following form of the Markov property: it holds that for any stopping time $T<\infty$, any non-negative $\F_t$-measurable functional $Z_t$, and any non-negative measurable functional $F$ on $\mathscr{D}^a$,
\begin{align}
\label{nMP}
  \bm{n}^a\left[Z_t\cdot F(X\circ \theta_T)\right]=\int \bm{n}^a\left[Z_t,\ X_T\in dx\right]\P_x^a[F(X)],
\end{align}
where $\theta$ is the shift operator and $\P_x^a$ is the distribution of the killed process upon $T_a$ (see, e.g., Theorem III.3.28 of \cite{Blu}).

\subsection{The renormalized zero resolvent}
For any $x\in \R$, the limit
\begin{align}
\label{h}
  h(x):=\lim_{q\to 0+}\left(r_q(0)-r_q(-x)\right)\qquad \text{for}\ q>0
\end{align}
exists and finite (see Theorem 1.1 of Takeda-Yano  \cite{TY}). We call the limiting function $h(x)$ the \emph{renormalized zero resolvent}. It is known that $h(x)$ is non-negative, continuous, and sub-additive function (see Theorem 1.1 of Takeda-Yano \cite{TY}). Further, we define
\begin{align}
\label{hgam}
  h^{(\gamma)}(x)&:=h(x)+\frac{\gamma}{\P_0[X_1^2]}x\qquad \text{for}\ -1\le \gamma\le 1,\\
  h^B(x)&:=h(x)+h(-x),\\
  h^C(x,y)&:=\frac{1}{h^B(x-y)}\left\{\begin{aligned}
&(h(y)+h(-x))h(x-y)+(h(x)+h(-y))h(y-x)\\
&\quad +(h(x)-h(y))(h(-y)-h(-x))-h(x-y)h(y-x)
\end{aligned}\right\}.
\end{align}
Note that when $\P_0[X_1^2]=\infty$, $h^{(\gamma)}$ is independent of $\gamma.$ Several formulas are known between these functions and hitting times: for $a,b\in \R$,
\begin{align}
 \label{P1}
  h^B(a)&=\P_0[L_{T_a}^0]=\frac{1}{\bm{n}^0(T_a<\infty)},\\
  h^C(a,b)&=\P_{0}[L_{T_a\wedge T_b}^0]=\frac{1}{\bm{n}^{0}(T_a\wedge T_b<\infty)},\\
  \label{P2}
  \P_x(T_a<T_b)&=\frac{h(b-a)+h(x-b)-h(x-a)}{h^B(a-b)},
\end{align}
(see Lemmas 3.5, 3.8, 6.1, and 6.3 of Takeda-Yano \cite{TY}).


\section{Exponential clock}
\label{Sec3}
We write $A_n:=\{a_1,...,a_n\}$ for distinct points $a_1,...,a_n\in \R$.

\subsection{Expectation}
Let us calculate the expectation $\P_x\left[\Gamma_{\bm{e}_q}^{(n)}\right]$. Since the local time $L_t^{a_k}$ is zero until the time $T_{a_k}$, we have
\begin{align*}
\label{exp1}
  \P_x\left[\Gamma_{\bm{e}_q}^{(n)}\right]&=\P_x\left[\int_0^\infty \Gamma_s^{(n)}qe^{-qs}ds\right]\\
  &=\P_x\left[\int_0^{T_{A_n}}qe^{-qs}ds\right]+\sum_{k=1}^n \P_x\left[\int_{T_{a_k}}^\infty \Gamma_s^{(n)}qe^{-qs}ds,\ T_{a_k}=T_{A_n}\right]\\
  &=1-\P_x\left[e^{-qT_{A_n}}\right]+\sum_{k=1}^n \P_x\left[\int_0^\infty \Gamma_{s+T_{a_k}}^{(n)}qe^{-q(s+T_{a_k})}ds,\ T_{a_k}=T_{A_n}\right]\\
  &=1-\P_x\left[e^{-qT_{A_n}}\right]+\sum_{k=1}^n \P_x\left[e^{-qT_{a_k}},\ T_{a_k}=T_{A_n}\right]\P_{a_k}\left[\Gamma_{\bm{e}_q}^{(n)}\right].
  \stepcounter{equation}\tag{\theequation}
\end{align*}
We now calculate the expectation $A_k^q:=\P_{a_k}\left[\Gamma_{\bm{e}_q}^{(n)}\right]$. By dividing into cases depending on the next point the process hits after $a_k$, we have
\begin{align*}
\label{exp2}
  A_k^q&=\P_{a_k}\left[\Gamma_{\bm{e}_q}^{(n)}\right]\\
  &=\P_{a_k}\left[\int_0^\infty \Gamma_s^{(n)}qe^{-qs}ds\right]\\
  &=\P_{a_k}\left[\int_0^{T_{A_n\setminus \{a_k\}}}\Gamma_s^{(n)}qe^{-qs}ds\right]+\sum_{\substack{i;\ i\le n\\ i\neq k}}\P_{a_k}\left[\int_{T_{a_i}}^\infty \Gamma_s^{(n)}qe^{-qs}ds,\ T_{a_i}=T_{A_n\setminus \{a_k\}}\right]\\
  &=\P_{a_k}\left[\int_0^{T_{A_n\setminus \{a_k\}}}e^{-\lambda_{a_k}L_s^{a_k}}qe^{-qs}ds\right]+\sum_{\substack{i;\ i\le n\\ i\neq k}}\P_{a_k}\left[\int_{0}^\infty \Gamma_{s+T_{a_i}}^{(n)}qe^{-q(s+T_{a_i})}ds,\ T_{a_i}=T_{A_n\setminus \{a_k\}}\right]\\
   &=\P_{a_k}\left[\int_0^{T_{A_n\setminus \{a_k\}}}e^{-\lambda_{a_k}L_s^{a_k}}qe^{-qs}ds\right]+\sum_{\substack{i;\ i\le n\\ i\neq k}}\P_{a_k}\left[e^{-qT_{a_i}}e^{-\lambda_{a_k}L_{T_{a_i}}^{a_k}},\ T_{a_i}=T_{A_n\setminus \{a_k\}}\right]\P_{a_i}\left[\Gamma_{\bm{e}_q}^{(n)}\right]\\
  &=:I_k^q+\sum_{\substack{i;\ i\le n\\ i\neq k}}J_{k,i}^qA_i^q.
  \stepcounter{equation}\tag{\theequation}
\end{align*}
Thus, we obtain the following simultaneous equations:
\begin{align}
  \begin{pmatrix}
    A_1^q\\
    A_2^q\\
    \vdots\\
    A_n^q
  \end{pmatrix}=
  \begin{pmatrix}
    0 & J_{1,2}^q & \cdots & J_{1,n-1}^q&J_{1,n}^q\\
    J_{2,1}^q & 0  & \cdots &J_{2,n-1}^q&J_{2,n}^q\\
    \vdots &\vdots &&\vdots & \vdots \\
    J_{n,1}^q & J_{n,2}^q &\cdots & J_{n,n-1}^q & 0
  \end{pmatrix}
  \begin{pmatrix}
    A_1^q\\
    A_2^q\\
    \vdots\\
    A_n^q
  \end{pmatrix}
  +
  \begin{pmatrix}
    I_1^q\\
    I_2^q\\
    \vdots\\
    I_n^q
  \end{pmatrix}.
\end{align}
We rewrite this as
\begin{align}
\label{simeq}
  \bm{a}_n^q=\mathbb{J}_n^q\bm{a}_n^q+\bm{i}_n^q.
\end{align}
To solve these simultaneous equations, we prove the following lemma.\\

\begin{lem}
\label{coeff}
  The matrix $\mathbb{E}_n-\mathbb{J}_n^q$ is invertible, where $\mathbb{E}_n$ denote the $n\times n$ identity matrix.
\end{lem}
\begin{proof}
  By \Levy-Desplanques theorem (see, e.g., Corollary 5.6.17 of \cite{HJ}), it suffices to show that $\mathbb{E}_n-\mathbb{J}_n^q$ is strictly diagonally dominant, that is
  \begin{align}
    1>\sum_{\substack{i;\ i\le n\\ i\neq k}}J_{k,i}^q\qquad \text{for}\ k=1,...,n.
  \end{align}
This inequality is obvious because 
\begin{align*}
  \sum_{\substack{i;\ i\le n\\ i\neq k}} J_{k,i}^q&=\sum_{\substack{i;\ i\le n\\ i\neq k}} \P_{a_k}\left[e^{-qT_{a_i}}e^{-\lambda_{a_k}L_{T_{a_i}}^{a_k}},\ T_{a_i}=T_{A_n\setminus \{a_k\}}\right]\\
  &<\sum_{\substack{i;\ i\le n\\ i\neq k}} \P_{a_k}\left(T_{a_i}=T_{A_n\setminus \{a_k\}}\right)\\
  &=1.
  \stepcounter{equation}\tag{\theequation}
\end{align*}
The proof is complete.
\end{proof}

Thanks to this Lemma, we can solve the simultaneous equations (\ref{simeq}). The solution is
\begin{align}
\label{exp3}
  \bm{a}_n^q=(\mathbb{E}_n-\mathbb{J}_n^q)^{-1}\bm{i}_n^q.
\end{align}
Therefore, by (\ref{exp1}), (\ref{exp2}), and (\ref{exp3}), $\P_x\left[\Gamma_{\bm{e}_q}^{(n)}\right]$ has been computed.

\subsection{Convergence}
Let us find the limit of $r_q(0)\P_x\left[\Gamma_{\bm{e}_q}^{(n)}\right]$ as $q\to 0+.$\\

\begin{prop}
\label{exp-conv}
It holds that
\begin{align}
  \lim_{q\to 0+}r_q(0)\P_{x}\left[\Gamma_{\bm{e}_q}^{(n)}\right]&=\varphi_{A_n}^{(0),\lambda_{a_1},...,\lambda_{a_n}}(x),\\
  \lim_{q\to 0+}r_q(0)\P_{X_t}\left[\Gamma_{\bm{e}_q}^{(n)}\right]&=\varphi_{A_n}^{(0),\lambda_{a_1},...,\lambda_{a_n}}(X_t)\qquad \text{in}\ L^1(\P_x).
\end{align}
\end{prop}

Before proving this proposition, we prove the following lemma.\\

\begin{lem}
\label{exp-lem}
It holds that
  \begin{align}
    \lim_{q\to 0}J_{k,i}^q&=J_{k,i},\\
    \label{limI}
  \lim_{q\to 0}r_q(0)I_{k}^q&=I_{k}^{(0)}.
  \end{align}
\end{lem}
\begin{proof}
 By the bounded convergence theorem, we have
  \begin{align*}
  \lim_{q\to 0}J_{k,i}^q&=\lim_{q\to 0}\P_{a_k}\left[e^{-qT_{a_i}}e^{-\lambda_{a_k}L_{T_{a_i}}^{a_k}},\ T_{a_i}=T_{A_n\setminus \{a_k\}}\right]\\
  &=\P_{a_k}\left[e^{-\lambda_{a_k}L_{T_{a_i}}^{a_k}},\ T_{a_i}=T_{A_n\setminus \{a_k\}}\right]\\
  &=J_{k,i}.
  \stepcounter{equation}\tag{\theequation}
\end{align*}

Next, we show the limit (\ref{limI}). To simplify notation, we write $A:=\{T_{A_n\setminus \{a_k\}}<\infty\}$, where $T_A$ is a first hitting time to $A$ with respect to excursion paths. By dividing the range of integration into excursion intervals, we have
\begin{align*}
\label{I-1}
  I_k^q&=\P_{a_k}\left[\int_0^{T_{A_n\setminus \{a_k\}}}\Gamma_s^{(n)}qe^{-qs}ds\right]\\
  &=\P_{a_k}\left[\sum_{l\le \sigma_A}\int_{\eta_{l-}^{a_k}}^{\eta_l^{a_k}\wedge T_{A_n\setminus \{a_k\}}}e^{-\lambda_{a_k}l}qe^{-qs}ds\right]\\
  &=\P_{a_k}\left[\int_{l\le \sigma_A}\int_{\mathscr{D}^{a_k}}\left(\int_{0}^{\eta_l^{a_k}\wedge T_{A_n\setminus \{a_k\}}-\eta_{l-}^{a_k}}e^{-\lambda_{a_k}l}qe^{-q(s+\eta_{l-}^{a_k})}ds\right)\bm{N}^{a_k}(dl\otimes de)\right]\\
  &=:\P_{a_k}\left[\int_{l\le \sigma_A}\int_{\mathscr{D}^{a_k}}F(\eta_{l-}^{a_k},X,e)\bm{N}^{a_k}(dl\otimes de)\right]\\
  &=\P_{a_k}\left[\left(\int_{l< \sigma_A}+\int_{l=\sigma_A}\right)\int_{\mathscr{D}^{a_k}}F(\eta_{l-}^{a_k},X,e)\bm{N}^{a_k}(dl\otimes de)\right],
  \stepcounter{equation}\tag{\theequation}
\end{align*}
where $\sigma_A$ is the first hitting time for the Poisson point process $(\epsilon_t^{a_k})_{t\ge 0}$:
\begin{align}
  \sigma_A:=\inf \left\{t\ge 0;\ \epsilon_t^{a_k}\in A\right\}.
\end{align}
Now, we define 
\begin{align}
\eta_l^{a_k,A^c}:=\int_{(0,l]}\int_{\mathscr{D}^{a_k}}\zeta(e)\bm{N}^{a_k}\left(ds\otimes de\cap \{(0,\infty)\times A^c\}\right).
\end{align}
The first term of the right hand side of (\ref{I-1}) is
\begin{align*}
\label{I-2}
  &\P_{a_k}\left[\int_{l< \sigma_A}\int_{\mathscr{D}^{a_k}}F(\eta_{l-}^{a_k},X,e)\bm{N}
  ^{a_k}(dl\otimes de)\right]\\
  &\qquad =\P_{a_k}\left[\int_{l< \sigma_A}\int_{\mathscr{D}^{a_k}}F(\eta_{l-}^{a_k,A^c},X,e)\bm{N}^{a_k}(dl\otimes de\cap \{(0,\infty)\times A^c\})\right]\\
  &\qquad =\P_{a_k}\left[\P_{a_k}\left[\int_{l< t}\int_{\mathscr{D}^{a_k}}F(\eta_{l-}^{a_k,A^c},X,e)\bm{N}^{a_k}(dl\otimes de\cap \{(0,\infty)\times A^c\})\right]_{t=\sigma_A}\right]\\
  &\qquad =\P_{a_k}\left[\int_0^{\sigma_A}\int_{\mathscr{D}^{a_k}}F(\eta_{l-}^{a,A^c},X,e)\bm{n}^{a_k}(de\cap A^c)dl\right]\\
  &\qquad =\P_{a_k}\left[\int_0^{\sigma_A}\int_{\mathscr{D}^{a_k}}\left(\int_0^{T_{a_k}}e^{-\lambda_{a_k}l}qe^{-q(s+\eta_{l-}^{a_k,A^c})}\right)\bm{n}^{a_k}(de\cap A^c)dl\right]\\
  &\qquad =\P_{a_k}\left[\int_0^{\sigma_A}e^{-\lambda_{a_k}l}e^{-q\eta_{l-}^{a_k,A^c}}dl\right]\bm{n}^{a_k}\left[\int_0^{T_{a_k}}qe^{-qs}ds,A^c\right]\\
  &\qquad =\P_{a_k}\left[\int_0^\infty\left(\int_0^{t}e^{-\lambda_{a_k}l}e^{-q\eta_{l-}^{a_k,A^c}}dl\right)\bm{n}^{a_k}(A)e^{-\bm{n}^{a_k}(A)t}dt\right]\bm{n}^{a_k}\left[1-e^{-qT_{a_k}},A^c\right]\\
  &\qquad =\left(\int_0^\infty\left(\int_0^{t}e^{-\lambda_{a_k}l}\P_{a_k}\left[e^{-q\eta_{l-}^{a_k,A^c}}\right]dl\right)\bm{n}^{a_k}(A)e^{-\bm{n}^{a_k}(A)t}dt\right)\bm{n}^{a_k}\left[1-e^{-qT_{a_k}},A^c\right],
\stepcounter{equation}\tag{\theequation}
\end{align*}
where in the second equality, we used independent increments of the Poisson point process, in the third equality, we used the compensation formula (see, e.g., Theorem 4.4 of \cite{Kyp}), and in the sixth equality, we used the facts that $\sigma_A$ and $\eta_{l-}^{a_k,A^c}$ are independent and that $\sigma_A$ has an exponential distribution with its parameter $\bm{n}^{a_k}(A)$ (see, e.g., Lemma 6.17 of \cite{Kyp}). Since $(\eta_l^{a_k})_{l\ge 0}$ is a subordinator with its \Levy\ measure $\bm{n}^{a_k}(T_{a_k}\in dx)$ and no drift, we have
\begin{align*}
  \P_{a_k}\left[e^{-q\eta_{l-}^{a_k,A^c}}\right]&=\P_{a_k}\left[e^{-q\eta_{l}^{a_k,A^c}}\right]\\
  &=\exp\left\{-l\int_{(0,\infty)}\left(1-e^{-qx}\right)\bm{n}^{a_k}\left(T_{a_k}\in dx\cap A^c\right)\right\}\\
  &=\exp \left\{-l\bm{n}^{a_k}\left[1-e^{-qT_{a_k}},A^c\right]\right\}.
  \stepcounter{equation}\tag{\theequation}
\end{align*}
Thus, we have
\begin{align*}
\label{I-3}
  (\ref{I-2})&=\left(\int_0^\infty\left(\int_0^{t}e^{-\lambda_{a_k}l}e^{-l\bm{n}^{a_k}[1-e^{-qT_{a_k}},A^c]}dl\right)\bm{n}^{a_k}(A)e^{-\bm{n}^{a_k}(A)t}dt\right)\bm{n}^{a_k}\left[1-e^{-qT_{a_k}},\ A^c\right]\\
  &=\frac{\bm{n}^{a_k}\left[1-e^{-qT_{a_k}},A^c\right]}{\bm{n}^{a_k}(A)+\bm{n}^{a_k}[1-e^{-qT_{a_k}},A^c]+\lambda_{a_k}}.
\stepcounter{equation}\tag{\theequation}
\end{align*}
The second term of the right hand side of (\ref{I-1}) is
\begin{align*}
\label{II-2}
&\P_{a_k}\left[\int_{l= \sigma_A}\int_{\mathscr{D}^{a_k}}F\left(\eta_{l-}^{a_k},X,e\right)\bm{N}^{a_k}(dl\otimes de)\right]\\
  &\qquad =\P_{a_k}\left[\int_{l= \sigma_A}\int_{\mathscr{D}^{a_k}}F\left(\eta_{l-}^{a_k,A^c},X,e\right)\bm{N}^{a_k}(dl\otimes de\cap \{(0,\infty)\times A\})\right]\\
  &\qquad =\P_{a_k}\left[F\left(\eta_{\sigma_A-}^{a_k,A^c},X,\epsilon_{\sigma_A}^{a_k}\right)\right]\\
  &\qquad =\P_{a_k}\left[\int_0^\infty \left(\int_{\mathscr{D}^{a_k}}F\left(\eta_{l-}^{a_k,A^c},X,e\right)\bm{n}^{a_k}(de|A)\right)\bm{n}^{a_k}(A)e^{-\bm{n}^{a_k}(A)l}dl\right]\\
  &\qquad =\P_{a_k}\left[\int_0^\infty \left(\int_{\mathscr{D}^{a_k}}F\left(\eta_{l-}^{a_k,A^c},X,e\right)\bm{n}^{a_k}(de\cap A)\right)e^{-\bm{n}^{a_k}(A)l}dl\right]\\
  &\qquad =\P_{a_k}\left[\int_0^\infty \left(\int_{\mathscr{D}^{a_k}}\left(\int_0^{T_{A_n\setminus \{a_k\}}}e^{-\lambda_{a_k}l}qe^{-q(s+\eta_{l-}^{a_k,A^c})}ds\right)\bm{n}^{a_k}(de\cap A)\right)e^{-\bm{n}^{a_k}(A)l}dl\right]\\
   &\qquad =\P_{a_k}\left[\int_0^\infty e^{-\lambda_{a_k}l}e^{-q\eta_{l-}^{a_k,A^c}}e^{-\bm{n}^{a_k}(A)l}dl\right]\bm{n}^{a_k}\left[\int_0^{T_{A_n}\setminus \{a_k\}}qe^{-qs}ds,A\right]\\
  &\qquad =\left(\int_0^\infty e^{-\lambda_{a_k}l}\P_{a_k}\left[e^{-q\eta_{l-}^{a_k,A^c}}\right]e^{-\bm{n}^{a_k}(A)l}dl\right)\bm{n}^{a_k}\left[1-e^{-qT_{A_n\setminus \{a_k\}}},A\right]\\
  &\qquad =\left(\int_0^\infty e^{-\lambda_{a_k}l}e^{-l\bm{n}^{a_k}[1-e^{-qT_{a_k}},\ A^c]}e^{-\bm{n}^{a_k}(A)l}dl\right)\bm{n}^{a_k}\left[1-e^{-qT_{A_n\setminus \{a_k\}}},A\right]\\
  &\qquad =\frac{\bm{n}^{a_k}\left[1-e^{-qT_{A_n\setminus \{a_k\}}},A\right]}{\bm{n}^{a_k}(A)+\bm{n}^{a_k}[1-e^{-qT_{a_k}},A^c]+\lambda_{a_k}},
  \stepcounter{equation}\tag{\theequation}
\end{align*}
where in the third equality, we used the facts that $\epsilon_{\sigma_A}^{a_k}$ has a distribution $\bm{n}^{a_k}(\cdot |A)$ and $\sigma_A$ has an exponential distribution with parameter $\bm{n}^{a_k}(A).$ By (\ref{A-2}), (\ref{nMP}), and (\ref{A-1}), we have
\begin{align*}
\label{II-3}
   &\bm{n}^{a_k}\left[1-e^{-qT_{a_k}},A^c\right]+\bm{n}^{a_k}\left[1-e^{-qT_{A_n\setminus \{a_k\}}},A\right]\\
  &\qquad =\bm{n}^{a_k}\left[1-e^{-qT_{a_k}},\ T_{A_n\setminus \{a_k\}}=\infty\right]+\bm{n}^{a_k}\left[1-e^{-qT_{A_n\setminus \{a_k\}}},\ T_{A_n\setminus \{a_k\}}<\infty\right]\\
  &\qquad =\bm{n}^{a_k}\left[1-e^{-qT_{a_k}}\right]-\bm{n}^{a_k}\left[1-e^{-qT_{a_k}},\ T_{A_n\setminus \{a_k\}}<\infty\right]+\bm{n}^{a_k}\left[1-e^{-qT_{A_n\setminus \{a_k\}}},\ T_{A_n\setminus \{a_k\}}<\infty\right]\\
  &\qquad =\frac{1}{r_q(0)}+\bm{n}^{a_k}\left[e^{-qT_{a_k}},\ T_{A_n\setminus \{a_k\}}<\infty\right]-\bm{n}^{a_k}\left[e^{-qT_{A_n\setminus \{a_k\}}},\ T_{A_n\setminus \{a_k\}}<\infty\right]\\
  &\qquad =\frac{1}{r_q(0)}+\sum_{\substack{i;\ i\le n\\ i\neq k}}\bm{n}^{a_k}\left[e^{-qT_{a_k}},\ T_{a_i}=T_{A_n\setminus \{a_k\}}<\infty\right]-\bm{n}^{a_k}\left[e^{-qT_{A_n\setminus \{a_k\}}},\ T_{A_n\setminus \{a_k\}}<\infty\right]\\
  &\qquad =\frac{1}{r_q(0)}+\sum_{\substack{i;\ i\le n\\ i\neq k}}\bm{n}^{a_k}\left[e^{-qT_{a_i}},\ T_{a_i}=T_{A_n\setminus \{a_k\}}<\infty\right]\P_{a_i}\left[e^{-qT_{a_k}}\right]\\
  &\qquad \qquad -\bm{n}^{a_k}\left[e^{-qT_{A_n\setminus \{a_k\}}},\ T_{A_n\setminus \{a_k\}}<\infty\right]\\
  &\qquad =\frac{1}{r_q(0)}+\sum_{\substack{i;\ i\le n\\ i\neq k}}\bm{n}^{a_k}\left[e^{-qT_{a_i}},\ T_{a_i}=T_{A_n\setminus \{a_k\}}<\infty\right]\frac{r_q(a_k-a_i)}{r_q(0)}\\
  &\qquad \qquad -\sum_{\substack{i;\ i\le n\\ i\neq k}}\bm{n}^{a_k}\left[e^{-qT_{a_i}},\ T_{a_i}=T_{A_n\setminus \{a_k\}}<\infty\right]\\
   &\qquad =\frac{1}{r_q(0)}+\sum_{\substack{i;\ i\le n\\ i\neq k}}\bm{n}^{a_k}\left[e^{-qT_{a_i}},\ T_{a_i}=T_{A_n\setminus \{a_k\}}<\infty\right]\frac{r_q(a_k-a_i)-r_q(0)}{r_q(0)}.
  \stepcounter{equation}\tag{\theequation}
\end{align*}
Therefore, by (\ref{h}), (\ref{I-1}), (\ref{I-3}), (\ref{II-2}), (\ref{II-3}), and (\ref{Ieq}), we have
\begin{align*}
  \lim_{q\to 0+}r_q(0)I_{k}^q&=\lim_{q\to \infty}\frac{1}{\bm{n}^{a_k}(A)+\bm{n}^{a_k}[1-e^{-qT_{a_k}},A^c]+\lambda_{a_k}}\\
  &\qquad \times r_q(0)\left\{\frac{1}{r_q(0)}+\sum_{\substack{i;\ i\le n\\ i\neq k}}\bm{n}^{a_k}\left[e^{-qT_{a_i}},\ T_{a_i}=T_{A_n\setminus \{a_k\}}<\infty\right]\frac{r_q(a_k-a_i)-r_q(0)}{r_q(0)}\right\}\\
  &=I_{k}^{(0)}.
    \stepcounter{equation}\tag{\theequation}
\end{align*}
The proof is complete.
\end{proof}

\begin{proof}[The proof of Proposition \ref{exp-conv}]
By Propositions 5.1 and 5.2 of Iba \cite{Iba}, we know that
\begin{align}
\lim_{q\to 0+}r_q(0)\left(1-\P_x\left[e^{-qT_{A_n}}\right]\right)&=\varphi_{A_n}^{(0)}(x),\\
\lim_{q\to 0+}r_q(0)\left(1-\P_{X_t}\left[e^{-qT_{A_n}}\right]\right)&=\varphi_{A_n}^{(0)}(X_t)\qquad \text{in}\ L^1(\P_x).
\end{align}
Thus, by (\ref{exp1}) and by Lemma \ref{exp-lem}, we obtain
\begin{align*}
   &\lim_{q\to 0+}r_q(0)\P_x\left[\Gamma_{\bm{e}_q}^{(n)}\right]\\
   &\qquad =\lim_{q\to 0+}r_q(0)\left(1-\P_x\left[e^{-qT_{A_n}}\right]\right)+\lim_{q\to 0+}r_q(0)\sum_{k=1}^n \P_x\left[e^{-qT_{a_k}},\ T_{a_k}=T_{A_n}\right]\P_{a_k}\left[\Gamma_{\bm{e}_q}^{(n)}\right]\\
   &\qquad =\varphi_{A_n}(x)+\sum_{k=1}^n \left\{\P_x\left(T_{a_k}=T_{A_k}\right)\cdot \lim_{q\to 0+}r_q(0)\P_{a_k}\left[\Gamma_{\bm{e}_q}^{(n)}\right]\right\}\\
   &\qquad =\varphi_{A_n}^{(0),\lambda_{a_1},...,\lambda_{a_n}}(x).
    \stepcounter{equation}\tag{\theequation}
\end{align*}
Similarly, we also obtain $L^1$-convergence. Therefore, the proof is complete.
\end{proof}

\begin{Rem}
\label{Rem1}
By (4.10) of Iba-Yano \cite{IY-3}, we know that
\begin{align}
\label{P7}
  \P_b\left[e^{-\lambda L_{T_c}^a}\right]=\frac{1+\lambda \{h(a-c)+h(b-a)-h(b-c)\}}{1+\lambda h^B(c-a)}\qquad \text{for}\ a,b,c\in \R.
\end{align}
Thus, we have
  \begin{align*}
  J_{k,i}&=\P_{a_k}\left[e^{-\lambda_{a_k}L_{T_{a_i}}^{a_k}},\ T_{a_i}=T_{A_n\setminus \{a_k\}}\right]\\
  &=\P_{a_k}\left[e^{-\lambda_{a_k}L_{T_{a_i}}^{a_k}}\right]-\P_{a_k}\left[e^{-\lambda_{a_k}L_{T_{a_i}}^{a_k}},\ T_{A_n\setminus \{a_k\}}<T_{a_i}\right]\\
  &=\P_{a_k}\left[e^{-\lambda_{a_k}L_{T_{a_i}}^{a_k}}\right]-\sum_{\substack{j;\ j\le n\\ j\neq k,i}}\P_{a_k}\left[e^{-\lambda_{a_k}L_{T_{a_i}}^{a_k}},\ T_{a_j}=T_{A_n\setminus \{a_k\}}\right]\\
   &=\P_{a_k}\left[e^{-\lambda_{a_k}L_{T_{a_i}}^{a_k}}\right]-\sum_{\substack{j;\ j\le n\\ j\neq k,i}}\P_{a_k}\left[e^{-\lambda_{a_k}L_{T_{a_j}}^{a_k}},\ T_{a_j}=T_{A_n\setminus \{a_k\}}\right]\P_{a_j}\left[e^{-\lambda_{a_k}L_{T_{a_i}}^{a_k}}\right]\\
   &=\P_{a_k}\left[e^{-\lambda_{a_k}L_{T_{a_i}}^{a_k}}\right]-\sum_{\substack{j;\ j\le n\\ j\neq k,i}}J_{k,j}\P_{a_j}\left[e^{-\lambda_{a_k}L_{T_{a_i}}^{a_k}}\right]\\
   &=\frac{1}{1+\lambda_{a_k}h^B(a_i-a_k)}-\sum_{\substack{j;\ j\le n\\ j\neq k,i}}J_{k,j}\cdot \frac{1+\lambda_{a_k}\{h(a_k-a_i)+h(a_j-a_k)-h(a_j-a_i)\}}{1+\lambda_{a_k}h^B(a_i-a_k)}\\
   &=:C_{k;i,k}-\sum_{\substack{j;\ j\le n\\ j\neq k,i}}J_{k,j}C_{k;i,j}.
   \stepcounter{equation}\tag{\theequation}
\end{align*}
We obtain the following simultaneous equations:
\begin{align}
  \begin{pmatrix}
    J_{k,1}\\
    J_{k,2}\\
    \vdots\\
    J_{k,n}
  \end{pmatrix}=
  \begin{pmatrix}
    C_{k;1,k}\\
    C_{k;2,k}\\
    \vdots\\
    C_{k;n,k}
  \end{pmatrix}
  -\begin{pmatrix}
    0 & C_{k;1,2} & \cdots & C_{k;1,n-1}&C_{k;1,n}\\
    C_{k;2,1} & 0  & \cdots &C_{k;2,n-1}&C_{k;2,n}\\
    \vdots &\vdots &&\vdots & \vdots \\
    C_{k;n,1} & C_{k;n,2} &\cdots & C_{k;n,n-1} & 0
  \end{pmatrix}
  \begin{pmatrix}
    J_{k,1}\\
    J_{k,2}\\
    \vdots\\
    J_{k,n}
  \end{pmatrix}.
\end{align}
We rewrite this as
  \begin{align}
    \bm{j}_k=\bm{c}_k-\mathbb{C}_k\bm{j}_k.
  \end{align}
We do not know whether the matrix $\mathbb{E}_{n-1}+\mathbb{C}_k$ is invertible.
\end{Rem}

\begin{Rem}
  For $k\neq n$, we have
  \begin{align*}
  \P_{x}(T_{a_k}=T_{A_{n}})&=\P_{x}(T_{a_k}<T_{a_n})-\P_{x}(T_{A_{n}\setminus \{a_k,a_n\}}<T_{a_k}<T_{a_n})\\
  &=\P_{x}(T_{a_k}<T_{a_n})-\sum_{\substack{i;\ i\le n-1\\ i\neq k}}\P_{x}(T_{a_i}=T_{A_{n}\setminus \{a_k,a_n\}}<T_{a_k}<T_{a_n})\\
  &=\P_{x}(T_{a_k}<T_{a_n})-\sum_{\substack{i;\ i\le n-1\\ i\neq k}}\P_{x}(T_{a_i}=T_{A_{n}})\P_{a_i}(T_{a_k}<T_{a_n})\\
  &=:p_{x;k,n}-\sum_{\substack{i;\ i\le n-1\\ i\neq k}}\P_{x}(T_{a_i}=T_{A_{n}})p_{i;k,n}.
  \stepcounter{equation}\tag{\theequation}
\end{align*}
We obtain the following simultaneous equations:
\begin{align}
  \begin{pmatrix}
    \P_{x}(T_{a_1}=T_{A_{n}})\\
    \P_{x}(T_{a_2}=T_{A_{n}})\\
    \vdots\\
    \P_{x}(T_{a_{n-1}}=T_{A_{n}})\\
    \P_{x}(T_{a_n}=T_{A_{n}})
  \end{pmatrix}=
  \begin{pmatrix}
    p_{x;1,n}\\
    p_{x;2,n}\\
    \vdots\\
    p_{x;n-1,n}\\
    1
  \end{pmatrix}
  -\begin{pmatrix}
    0 & p_{2;1,n} & \cdots & p_{n-1;1,n}&0\\
    p_{1;2,n} & 0  & \cdots &p_{n-1;2,n}&0\\
    \vdots &\vdots &&\vdots & \vdots \\
    p_{1;n-1,n} & p_{2;n-1,n} &\cdots & 0 & 0\\
    1&1&\cdots &1&0
  \end{pmatrix}
  \begin{pmatrix}
    \P_{x}(T_{a_1}=T_{A_{n}})\\
    \P_{x}(T_{a_2}=T_{A_{n}})\\
    \vdots\\
    \P_{x}(T_{a_{n-1}}=T_{A_{n}})\\
    \P_{x}(T_{a_n}=T_{A_{n}})
  \end{pmatrix}.
\end{align}
We rewrite this as
  \begin{align}
    \bm{p}_n(x)=\bm{p}_n'-\mathbb{P}_n\bm{p}_n(x).
  \end{align}
We do not know whether the matrix $\mathbb{E}_{n}+\mathbb{P}_n$ is invertible.
\end{Rem}

\subsection{Penalization}
Although the proof of the following proposition is parallel to Theorem 3.4 of Iba-Yano \cite{IY-3}, we give the proof for completeness of this paper.\\

\begin{thm}
\label{penal-1}
  The process $\left(\varphi_{A_n}^{(0),\lambda_{a_1},...,\lambda_{a_n}}(X_t)\Gamma_t^{(n)}\right)_{t\ge 0}$ is a martingale. Moreover, it holds that
  \begin{align}
    \lim_{q\to 0+}\P_x\left[F_s\cdot \Gamma_{\bm{e}_q}^{(n)}\right]=\P_x\left[F_s\cdot \varphi_{A_n}^{(0),\lambda_{a_1},...,\lambda_{a_n}}(X_s)\Gamma_s^{(n)}\right]
  \end{align}
  for all bounded $\F_t$-measurable functionals $F_t$.
\end{thm}
\begin{proof}
  We define
  \begin{align}
    N_t^{q,(n)}&:=r_q(0)\P_x\left[\Gamma_{\bm{e}_q}^{(n)};\ t<\bm{e}_q|\F_t\right],\\
    M_t^{q,(n)}&:=r_q(0)\P_x\left[\Gamma_{\bm{e}_q}^{(n)}|\F_t\right].
  \end{align}
  By the lack of memory property of an exponential distribution, we have
  \begin{align*}
    N_t^{q,(n)}&=r_q(0)\P_x\left[\Gamma_{\bm{e}_q}^{(n)};\ t<\bm{e}_q|\F_t\right]\\
    &=r_q(0)\P_x^{\F_t}\left[\Gamma_{\bm{e}_q}^{(n)}|\ t<\bm{e}_q\right]\P_x(t<\bm{e}_q)\\
&=r_q(0)\Gamma_t^{(n)}\P_{X_t}\left[\Gamma_{\bm{e}_q}^{(n)}\right]\cdot e^{-qt},
    \stepcounter{equation}\tag{\theequation}
  \end{align*}
  where $\P_x^{\F_t}[\cdot]$ denotes the conditional expectation for $\P_x$ given $\F_t$. By Proposition \ref{exp-conv}, we have
  \begin{align*}
  \lim_{q\to 0+}N_t^{q,(n)}&=\lim_{q\to 0+}r_q(0)\Gamma_t^{(n)}\P_{X_t}\left[\Gamma_{\bm{e}_q}^{(n)}\right]\cdot e^{-qt}\\
  &=\varphi_{A_n}^{(0),\lambda_{a_1},...,\lambda_{a_n}}(X_t)\Gamma_t^{(n)}\qquad \text{a.s.\ and\ in}\ L^1(\P_x).
    \stepcounter{equation}\tag{\theequation}
  \end{align*}
  Since 
  \begin{align}
    \lim_{q\to 0+}qr_q(0)=0
  \end{align}
  by Lemma 15.5 of Tsukada \cite{Tukada}, we have
  \begin{align*}
    \lim_{q\to 0+}\left(M_t^{q,(n)}-N_t^{q,(n)}\right)&=\lim_{q\to 0+}r_q(0)\P_{x}\left[\Gamma_{\bm{e}_q}^{(n)},\ \bm{e}_q\le t|\F_t\right]\\
    &=\lim_{q\to 0+}r_q(0)\int_0^t \Gamma_u^{(n)}\cdot qe^{-qu}du\\
    &\le \lim_{q\to 0+}qr_q(0)\cdot t\\
    &=0.
    \stepcounter{equation}\tag{\theequation}
  \end{align*}
  Thus, we obtain
  \begin{align}
    \lim_{q\to 0+}M_t^{q,(n)}=\lim_{q\to 0+}N_t^{q,(n)}=\varphi_{A_n}^{(0),\lambda_{a_1},...,\lambda_{a_n}}(X_t)\Gamma_t^{(n)}\qquad \text{a.s.\ and\ in}\ L^1(\P_x).
  \end{align}
  Since $(M_t^{q,(n)})_{t\ge 0}$ is a non-negative martingale, its $L^1$-limit $(\varphi_{A_n}^{(0),\lambda_{a_1},...,\lambda_{a_n}}(X_t)\Gamma_t^{(n)})_{t\ge 0}$ is also a non-negative martingale (see, e.g., Proposition 1.3 of \cite{CW}).

Finally, we obtain by the $L^1$-convergence,
\begin{align*}
  \lim_{q\to 0+}r_q(0)\P_x[F_s\cdot \Gamma_{\bm{e}_q}^{(n)}]&=\lim_{q\to 0+}r_q(0)\P_x[F_s\cdot \P_x[\Gamma_{\bm{e}_q}^{(n)}|\F_s]]\\
  &=\lim_{q\to 0+}\P_x[F_s\cdot M_s^{q,(n)}]\\
  &=\P_x[F_s\cdot \varphi_{A_n}^{(0),\lambda_{a_1},...,\lambda_{a_n}}(X_s)\Gamma_s^{(n)}].
  \stepcounter{equation}\tag{\theequation}
\end{align*}
for a bounded $\F_t$-measurable functional $F_t$. The proof is complete.
\end{proof}


\section{One-point hitting time clock}
\label{Sec4}
\subsection{Expectation}
Let us calculate the expectation $\P_x\left[\Gamma_{T_b}^{(n)}\right]$. We have
\begin{align*}
\label{hit-1}
\P_x\left[\Gamma_{T_b}^{(n)}\right]&=\P_x(T_b<T_{A_n})+\P_x\left[\Gamma_{T_b}^{(n)},\ T_{A_n}<T_b\right]\\
&=\P_x(T_b<T_{A_n})+\sum_{k=1}^n\P_x\left[\Gamma_{T_b}^{(n)},\ T_{a_k}=T_{A_n}\wedge T_b\right]\\
&=\P_x(T_b<T_{A_n})+\sum_{k=1}^n\P_x(T_{a_k}=T_{A_n}\wedge T_b)\P_{a_k}\left[\Gamma_{T_b}^{(n)}\right].
 \stepcounter{equation}\tag{\theequation}
\end{align*}
We now calculate the expectation $A_k^b:=\P_{a_k}\left[\Gamma_{T_b}^{(n)}\right]$. We have
\begin{align*}
\label{hit-2}
  A_k^b&=\P_{a_k}\left[\Gamma_{T_b}^{(n)}\right]\\
  &=\sum_{\substack{i;\ i\le n\\ i\neq k}}\P_{a_k}\left[\Gamma_{T_b}^{(n)},\ T_{a_i}=T_{A_n\setminus \{a_k\}}\wedge T_b\right]+\P_{a_k}\left[\Gamma_{T_b}^{(n)},\ T_b<T_{A_n\setminus \{a_k\}}\right]\\
  &=\sum_{\substack{i;\ i\le n\\ i\neq k}}\P_{a_k}\left[e^{-\lambda_{a_k}L_{T_{a_i}}^{a_k}},\ T_{a_i}=T_{A_n\setminus \{a_k\}}\wedge T_b\right]\P_{a_i}\left[\Gamma_{T_b}^{(n)}\right]+\P_{a_k}\left[e^{-\lambda_{a_k}L_{T_b}^{a_k}},\ T_b<T_{A_n\setminus \{a_k\}}\right]\\
  &=:\sum_{\substack{i;\ i\le n\\ i\neq k}}J_{k,i}^bA_i^b+I_{k}^b.
  \stepcounter{equation}\tag{\theequation}
\end{align*}
Thus, we obtain the following simultaneous equations:
\begin{align}
  \begin{pmatrix}
    A_1^b\\
    A_2^b\\
    \vdots\\
    A_n^b
  \end{pmatrix}=
  \begin{pmatrix}
    0 & J_{1,2}^b & \cdots & J_{1,n-1}^b&J_{1,n}^b\\
    J_{2,1}^b & 0  & \cdots &J_{2,n-1}^b&J_{2,n}^b\\
    \vdots &\vdots &&\vdots & \vdots \\
    J_{n,1}^b & J_{n,2}^b &\cdots & J_{n,n-1}^b & 0
  \end{pmatrix}
  \begin{pmatrix}
    A_1^b\\
    A_2^b\\
    \vdots\\
    A_n^b
  \end{pmatrix}
  +
  \begin{pmatrix}
    I_1^b\\
    I_2^b\\
    \vdots\\
    I_n^b
  \end{pmatrix}.
\end{align}
We rewrite this as
\begin{align}
\label{hit-3}
  \bm{a}_n^b=\mathbb{J}_n^b\bm{a}_n^b+\bm{i}_n^b.
\end{align}
Since strictly diagonally dominance can be shown in the same way as in Lemma \ref{coeff}, the solution of these simultaneous equations is
\begin{align}
  \bm{a}_n^b=(\mathbb{E}_n-\mathbb{J}_n^b)^{-1}\bm{i}_n^b.
\end{align}
Therefore, by (\ref{hit-1}), (\ref{hit-2}), and (\ref{hit-3}), $\P_x\left[\Gamma_{T_b}^{(n)}\right]$has been computed.

\subsection{Convergence}
Let us find the limit of $h^B(b)\P_x\left[\Gamma_{T_b}^{(n)}\right]$ as $b\to \pm \infty.$\\

\begin{prop}
\label{hit}
  It holds that
  \begin{align}
    \lim_{b\to \pm \infty}h^B(b)\P_{x}\left[\Gamma_{T_b}^{(n)}\right]&=\varphi_{A_n}^{(\pm 1),\lambda_{a_1},...,\lambda_{a_n}}(x),\\
  \lim_{b\to \pm \infty}h^B(b)\P_{X_t}\left[\Gamma_{T_b}^{(n)}\right]&=\varphi_{A_n}^{(\pm 1),\lambda_{a_1},...,\lambda_{a_n}}(X_t)\qquad \text{in}\ L^1(\P_x).
  \end{align}
\end{prop}

Before proving this proposition, we prove the following lemma.

\begin{lem}
\label{hit-lem}
  It holds that
  \begin{align}
    \lim_{b\to \pm \infty}J_{k,i}^b&=J_{k,i},\\
    \label{hitlim-2}
  \lim_{b\to \pm \infty}h^B(b)I_k^b&=I_k^{(\pm 1)}.
  \end{align}
\end{lem}
\begin{proof}
  By the bounded convergence theorem, we have
  \begin{align*}
  \lim_{b\to \pm \infty}J_{k,i}^b&=\lim_{b\to \pm \infty}\P_{a_k}\left[e^{-\lambda_{a_k}L_{T_{a_i}}^{a_k}},\ T_{a_i}=T_{A_n\setminus \{a_k\}}\wedge T_b\right]\\
  &=\P_{a_k}\left[e^{-\lambda_{a_k}L_{T_{a_i}}^{a_k}},\ T_{a_i}=T_{A_n\setminus \{a_k\}}\right]\\
  &=J_{k,i}.
  \stepcounter{equation}\tag{\theequation}
\end{align*}

Next, we show the limit (\ref{hitlim-2}). We have
\begin{align*}
  I_k^b&=\P_{a_k}\left[e^{-\lambda_{a_k}L_{T_b}^{a_k}},\ T_b<T_{A_n\setminus \{a_k\}}\right]\\
  &=\P_{a_k}\left[e^{-\lambda_{a_k}L_{T_b}^{a_k}}\right]-\sum_{\substack{i;\ i\le n\\ i\neq k}}\P_{a_k}\left[e^{-\lambda_{a_k}L_{T_b}^{a_k}},\ T_{a_i}=T_{A_n\setminus \{a_k\}}\wedge T_b\right]\\
    &=\P_{a_k}\left[e^{-\lambda_{a_k}L_{T_b}^{a_k}}\right]-\sum_{\substack{i;\ i\le n\\ i\neq k}}\P_{a_k}\left[e^{-\lambda_{a_k}L_{T_{a_i}}^{a_k}},\ T_{a_i}=T_{A_n\setminus \{a_k\}}\wedge T_b\right]\P_{a_i}\left[e^{-\lambda_{a_k}L_{T_b}^{a_k}}\right].
    \stepcounter{equation}\tag{\theequation}
\end{align*}
By (4.17) of Iba-Yano \cite{IY-3}, we know that
\begin{align}
  \lim_{b\to \pm \infty}h^B(b)\P_{a}\left[e^{-\lambda L_{T_b}^{c}}\right]=\frac{1}{\lambda}+h^{(\pm 1)}(a-c).
  \end{align}
  Therefore, we obtain
  \begin{align*}
  \lim_{b\to \pm \infty}h^B(b)I_k^b&=\lim_{b\to \pm \infty}h^B(b)\P_{a_k}\left[e^{-\lambda_{a_k}L_{T_b}^{a_k}}\right]\\
  &\qquad -\sum_{\substack{i;\ i\le n\\ i\neq k}}\left\{\P_{a_k}\left[e^{-\lambda_{a_k}L_{T_{a_i}}^{a_k}},\ T_{a_i}=T_{A_n\setminus \{a_k\}}\right]\cdot \lim_{b\to \pm \infty}h^B(b)\P_{a_i}\left[e^{-\lambda_{a_k}L_{T_b}^{a_k}}\right]\right\}\\
  &=\frac{1}{\lambda_{a_k}}-\sum_{\substack{i;\ i\le n\\ i\neq k}}J_{k,i}\left(\frac{1}{\lambda_{a_k}}+h^{(\pm 1)}(a_i-a_k)\right)\\
  &=I_{k}^{(\pm 1)}.
   \stepcounter{equation}\tag{\theequation}
\end{align*}
The proof is complete.
\end{proof}

\begin{proof}[The proof of Proposition \ref{hit}]
By Propositions 5.1 and 5.2 of Iba \cite{Iba}, we know that
\begin{align}
\lim_{b\to \pm \infty}h^B(b)\P_x(T_b<T_{A_n})&=\varphi_{A_n}^{(\pm 1)}(x),\\
\lim_{b\to \pm \infty}h^B(b)\P_{X_t}(T_b<T_{A_n})&=\varphi_{A_n}^{(\pm 1)}(X_t)\qquad \text{in}\ L^1(\P_x).
\end{align}
Thus, by (\ref{hit-1}) and by Lemma \ref{hit-lem}, we obtain
\begin{align*}
  &\lim_{b\to \pm \infty}h^B(b)\P_x\left[\Gamma_{T_b}^{(n)}\right]\\
   &\qquad =\lim_{b\to \pm \infty}h^B(b)\P_x(T_b<T_{A_n})+\lim_{b\to \pm \infty}h^B(b)\sum_{k=1}^n\P_x(T_{a_k}=T_{A_n}\wedge T_b)\P_{a_k}\left[\Gamma_{T_b}^{(n)}\right]\\
   &\qquad =\varphi_{A_n}^{(\pm 1)}(x)+\sum_{k=1}^n \left\{\P_x\left(T_{a_k}=T_{A_k}\right)\cdot \lim_{b\to \pm \infty}h^B(b)\P_{a_k}\left[\Gamma_{T_b}^{(n)}\right]\right\}\\
   &\qquad =\varphi_{A_n}^{(\pm 1),\lambda_{a_1},...,\lambda_{a_n}}(x).
    \stepcounter{equation}\tag{\theequation}
\end{align*}
Similarly, we also obtain $L^1$-convergence. Therefore, the proof is complete.
\end{proof}

\subsection{Penalization}
\begin{thm}
\label{penal-2}
  The process $\left(\varphi_{A_n}^{(\pm 1),\lambda_{a_1},...,\lambda_{a_k}}(X_t)\Gamma_t^{(n)}\right)_{t\ge 0}$ is a martingale. Moreover, it holds that
  \begin{align}
    \lim_{b\to \pm \infty}\P_x\left[F_s\cdot \Gamma_{T_b}^{(n)}\right]=\P_x\left[F_s\cdot \varphi_{A_n}^{(\pm 1),\lambda_{a_1},...,\lambda_{a_n}}(X_s)\Gamma_s^{(n)}\right]
  \end{align}
  for all bounded $\F_t$-measurable functionals $F_t$.
\end{thm}

The proof is almost the same as that of Theorem \ref{penal-1}, based on Proposition \ref{hit} and so we omit it.


\section{Two-point hitting time clock}
\label{Sec5}
\subsection{Expectation}
Let us calculate the expectation $\P_x\left[\Gamma_{T_c\wedge T_{-d}}^{(n)}\right]$. We have
\begin{align*}
  &\P_x\left[\Gamma_{T_c}^{(n)},\ T_{c}<T_{-d}\right]\\
  &\qquad =\P_x\left[\Gamma_{T_c}^{(n)}\right]-\P_x\left[\Gamma_{T_c}^{(n)},\ T_{-d}<T_{c}\right]\\
  &\qquad =\P_x\left[\Gamma_{T_c}^{(n)}\right]-\P_x\left[\Gamma_{T_{-d}}^{(n)},\ T_{-d}<T_{c}\right]\P_{-d}\left[\Gamma_{T_c}^{(n)}\right]\\
  &\qquad =\P_x\left[\Gamma_{T_c}^{(n)}\right]-\left(\P_x\left[\Gamma_{T_{-d}}^{(n)}\right]-\P_x\left[\Gamma_{T_{-d}}^{(n)},\ T_{c}<T_{-d}\right]\right)\P_{-d}\left[\Gamma_{T_c}^{(n)}\right]\\
    &\qquad =\P_x\left[\Gamma_{T_c}^{(n)}\right]-\left(\P_x\left[\Gamma_{T_{-d}}^{(n)}\right]-\P_x\left[\Gamma_{T_c}^{(n)},\ T_{c}<T_{-d}\right]\P_c\left[\Gamma_{T_{-d}}^{(n)}\right]\right)\P_{-d}\left[\Gamma_{T_c}^{(n)}\right].
     \stepcounter{equation}\tag{\theequation}
\end{align*}
Thus, we obtain
\begin{align}
   \P_x\left[\Gamma_{T_c}^{(n)},\ T_{c}<T_{-d}\right]=\frac{\P_x\left[\Gamma_{T_c}^{(n)}\right]-\P_x\left[\Gamma_{T_{-d}}^{(n)}\right]\P_{-d}\left[\Gamma_{T_c}^{(n)}\right]}{1-\P_c\left[\Gamma_{T_{-d}}^{(n)}\right]\P_{-d}\left[\Gamma_{T_c}^{(n)}\right]}.
\end{align}
Similarly, we have
\begin{align}
   \P_x\left[\Gamma_{T_{-d}}^{(n)},\ T_{-d}<T_{c}\right]=\frac{\P_x\left[\Gamma_{T_{-d}}^{(n)}\right]-\P_x\left[\Gamma_{T_{c}}^{(n)}\right]\P_{c}\left[\Gamma_{T_{-d}}^{(n)}\right]}{1-\P_{-d}\left[\Gamma_{T_{c}}^{(n)}\right]\P_{c}\left[\Gamma_{T_{-d}}^{(n)}\right]}.
\end{align}
Therefore, since
\begin{align}
  \P_x\left[\Gamma_{T_c\wedge T_{-d}}^{(n)}\right]=\P_x\left[\Gamma_{T_c}^{(n)},\ T_{c}<T_{-d}\right]+\P_x\left[\Gamma_{T_{-d}}^{(n)},\ T_{-d}<T_{c}\right],
\end{align}
the expectation $\P_x\left[\Gamma_{T_c\wedge T_{-d}}^{(n)}\right]$ has been computed.

\subsection{Convergence}
Let us find the limit of $h^C(c,-d)\P_x\left[\Gamma_{T_c\wedge T_{-d}}^{(n)}\right]$ as $(c,-d)\stackrel{(\gamma)}{\to} \pm \infty.$\\

\begin{prop}
\label{two-hit}
  It holds that
  \begin{align}
    \lim_{(c,d)\stackrel{(\gamma)}{\to}\infty}h^C(c,-d)\P_x\left[\Gamma_{T_c\wedge T_{-d}}^{(n)}\right]&=\varphi_{A_n}^{(\gamma),\lambda_{a_1},...,\lambda_{a_n}}(x),\\
    \lim_{(c,d)\stackrel{(\gamma)}{\to}\infty}h^C(c,-d)\P_{X_t}\left[\Gamma_{T_c\wedge T_{-d}}^{(n)}\right]&=\varphi_{A_n}^{(\gamma),\lambda_{a_1},...,\lambda_{a_n}}(X_t)\qquad \text{in}\ L^1(\P_x).
  \end{align}
\end{prop}
\begin{proof}
  By p.202 of Iba-Yano \cite{IY-3}, we know that
\begin{align}
  \lim_{(c,d)\stackrel{(\gamma)}{\to}\infty}\P_{-d}(T_c<T_{a})=0\qquad \text{for}\ a\in \R.
\end{align}
  Since 
    \begin{align*}
  0&\le\lim_{(c,d)\stackrel{(\gamma)}{\to}\infty}\P_{-d}\left[\Gamma_{T_c}^{(n)}\right]\\
  &=\lim_{(c,d)\stackrel{(\gamma)}{\to}\infty}\left(\P_{-d}(T_c<T_{A_n})+\sum_{k=1}^n\P_{-d}(T_{a_k}=T_{A_n}\wedge T_c)\P_{a_k}\left[\Gamma_{T_c}^{(n)}\right]\right)\\
  &=\lim_{(c,d)\stackrel{(\gamma)}{\to}\infty}\P_{-d}(T_c<T_{A_n})\\
  &\le \lim_{(c,d)\stackrel{(\gamma)}{\to}\infty}\P_{-d}(T_c<T_{a_1})\\
  &=0,
  \stepcounter{equation}\tag{\theequation}
\end{align*}
we have
\begin{align*}
  &\lim_{(c,d)\stackrel{(\gamma)}{\to}\infty}h^C(c,-d)\P_x\left[\Gamma_{T_c}^{(n)},\ T_{c}<T_{-d}\right]\\
  &\qquad =\lim_{(c,d)\stackrel{(\gamma)}{\to}\infty}\frac{\frac{h^C(c,-d)}{h^B(c)}\cdot h^B(c)\P_x\left[\Gamma_{T_c}^{(n)}\right]-\frac{h^C(c,-d)}{h^B(-d)}\cdot h^B(-d)\P_x\left[\Gamma_{T_{-d}}^{(n)}\right]\P_{-d}\left[\Gamma_{T_c}^{(n)}\right]}{1-\P_c\left[\Gamma_{T_{-d}}^{(n)}\right]\P_{-d}\left[\Gamma_{T_c}^{(n)}\right]}\\
    &\qquad =\lim_{(c,d)\stackrel{(\gamma)}{\to}\infty}\frac{h^C(c,-d)}{h^B(c)}\cdot h^B(c)\P_x\left[\Gamma_{T_c}^{(n)}\right]\\
  &\qquad =\frac{1+\gamma}{2}\cdot \varphi_{A_n}^{(+ 1),\lambda_{a_1},...,\lambda_{a_n}}(x).
  \stepcounter{equation}\tag{\theequation}
\end{align*}
Similarly, we have
\begin{align}
  \lim_{(c,d)\stackrel{(\gamma)}{\to}\infty}h^C(c,-d)\P_x\left[\Gamma_{T_{-d}}^{(n)},\ T_{-d}<T_{c}\right]=\frac{1-\gamma}{2}\cdot \varphi_{A_n}^{(-1),\lambda_{a_1},...,\lambda_{a_n}}(x).
\end{align}
Therefore, we obtain
\begin{align*}
  &\lim_{(c,d)\stackrel{(\gamma)}{\to}\infty}h^C(c,-d)\P_x\left[\Gamma_{T_{-d}}^{(n)}\right]\\
  &\qquad =\lim_{(c,d)\stackrel{(\gamma)}{\to}\infty}h^C(c,-d)\P_x\left[\Gamma_{T_c}^{(n)},\ T_{c}<T_{-d}\right]+\lim_{(c,d)\stackrel{(\gamma)}{\to}\infty}h^C(c,-d)\P_x\left[\Gamma_{T_{-d}}^{(n)},\ T_{-d}<T_{c}\right]\\
  &\qquad =\frac{1+\gamma}{2}\cdot \varphi_{A_n}^{(+ 1),\lambda_{a_1},...,\lambda_{a_n}}(x)+\frac{1-\gamma}{2}\cdot \varphi_{A_n}^{(-1),\lambda_{a_1},...,\lambda_{a_n}}(x)\\
  &\qquad =\varphi_{A_n}^{(\gamma),\lambda_{a_1},...,\lambda_{a_n}}(x).
  \stepcounter{equation}\tag{\theequation}
\end{align*}
Similarly, we also obtain $L^1$-convergence. Therefore, the proof is complete.
\end{proof}

\subsection{Penalization}
\begin{thm}
\label{penal-3}
  The process $\left(\varphi_{A_n}^{(\gamma),\lambda_{a_1},...,\lambda_{a_k}}(X_t)\Gamma_t^{(n)}\right)_{t\ge 0}$ is a martingale. Moreover, it holds that
  \begin{align}
    \lim_{(c,d)\stackrel{(\gamma)}{\to}\infty}\P_x\left[F_s\cdot \Gamma_{T_{c}\wedge T_{-d}}^{(n)}\right]=\P_x\left[F_s\cdot \varphi_{A_n}^{(\gamma),\lambda_{a_1},...,\lambda_{a_n}}(X_s)\Gamma_s^{(n)}\right]
  \end{align}
  for all bounded $\F_t$-measurable functionals $F_t$.
\end{thm}

The proof is almost the same as that of Theorem \ref{penal-1}, based on Proposition \ref{two-hit} and so we omit it.


\section{Inverse local time clock}
\label{Sec6}
\subsection{Convergence}
\begin{prop}
\label{inv}
  It holds that
  \begin{align}
    \lim_{b\to \pm \infty}h^B(b)\P_{x}\left[\Gamma_{\eta_u^b}^{(n)}\right]&=\varphi_{A_n}^{(\pm 1),\lambda_{a_1},...,\lambda_{a_n}}(x),\\
  \lim_{b\to \pm \infty}h^B(b)\P_{X_t}\left[\Gamma_{\eta_u^b}^{(n)}\right]&=\varphi_{A_n}^{(\pm 1),\lambda_{a_1},...,\lambda_{a_n}}(X_t)\qquad \text{in}\ L^1(\P_x).
  \end{align}
\end{prop}
\begin{proof}
  By the strong Markov property, we have
  \begin{align}
\P_x\left[\Gamma_{\eta_u^b}^{(n)}\right]=\P_x\left[\Gamma_{T_b}^{(n)}\right]\P_b\left[\Gamma_{\eta_u^b}^{(n)}\right].
\end{align}
We now calculate the limit of $\P_b\left[\Gamma_{\eta_u^b}^{(n)}\right]$ as $b\to \pm \infty.$ By Campbell formula (see, e.g., Theorem 2.7 of \cite{Kyp}), we have
\begin{align*}
\label{P4}
  \P_b\left[\Gamma_{\eta_u^b}^{(n)}\right]&=\P_b\left[\exp \left(-\sum_{v\le u}\left(\lambda_{a_1}L_{T_b}^{a_1}(\epsilon_v^b)+\cdots +\lambda_b L_{T_b}^b (\epsilon_v^b\right)\right)\right]\\
  &=\exp \left\{-\int_{(0,u]\times \mathscr{D}^b}\left(1-e^{-\lambda_{a_1}L_{T_b}^{a_n}(e)-\cdots -\lambda_{a_n}L_{T_b}^{a_n}(e)}\right)dt\otimes \bm{n}^{b}(de)\right\}\\
  &=\exp \left\{-u \bm{n}^b\left[1-\Gamma_{T_b}^{(n)}\right]\right\}.
  \stepcounter{equation}\tag{\theequation}
\end{align*}
Since by Lemma 3.7 of Takeda-Yano \cite{TY},
\begin{align*}
   0&\le \bm{n}^b\left[1-\Gamma_{T_b}^{(n)}\right]\\
   &=\sum_{i=1}^n\bm{n}^b\left[1-\Gamma_{T_b}^{(n)},\ T_{a_i}=T_{A_n}\wedge T_b\right]\\
   & \le 2\sum_{i=1}^n\bm{n}^b(T_{a_i}<T_b)\\
   & =2\sum_{i=1}^n \frac{1}{h^B(a_i-b)}\to 0
  \stepcounter{equation}\tag{\theequation}
\end{align*}
as $b\to \pm \infty$, we have
\begin{align}
  \lim_{b\to \pm \infty}\P_b\left[\Gamma_{\eta_u^b}^{(n)}\right]=1.
\end{align}
Thus, we obtain
\begin{align*}
\lim_{b\to \pm \infty}h^B(b)\P_x\left[\Gamma_{\eta_u^b}^{(n)}\right]&=\lim_{b\to \pm \infty}h^B(b)\P_x\left[\Gamma_{T_b}^{(n)}\right]\P_b\left[\Gamma_{\eta_u^b}^{(n)}\right]\\
&=\varphi_{A_n}^{(\pm 1),\lambda_{a_1},...,\lambda_{a_n}}(x).
\stepcounter{equation}\tag{\theequation}
\end{align*}
Similarly, we also obtain $L^1$-convergence. Therefore, the proof is complete.
\end{proof}

\subsection{Penalization}
\begin{thm}
\label{penal-4}
  It holds that
  \begin{align}
    \lim_{b\to \pm \infty}\P_x\left[F_s\cdot \Gamma_{\eta_u^b}^{(n)}\right]=\P_x\left[F_s\cdot \varphi_{A_n}^{(\pm 1),\lambda_{a_1},...,\lambda_{a_n}}(X_s)\Gamma_s^{(n)}\right]
  \end{align}
  for all bounded $\F_t$-measurable functionals $F_t$.
\end{thm}

The proof is almost the same as that of Theorem \ref{penal-1}, based on Proposition \ref{inv} and so we omit it.


\section{The case of $n=2$}
\label{Sec7}
\subsection{Exponential clock}
By (\ref{P1}) and (\ref{P7}), we have
\begin{align}
\label{P3}
   J_{k,i}&=\P_{a_k}\left[e^{-\lambda_{a_k}L_{T_{a_i}}^{a_k}}\right]=\frac{1}{1+\lambda_{a_k}h^B(a_k-a_i)},\\
  I_{1}&=\frac{1-h(a_2-a_1)n^{a_1}(T_{a_2}<\infty)}{n^{a_1}(T_{a_2}<\infty)+\lambda_{a_1}}=\frac{h(a_1-a_2)}{1+\lambda_{a_1}h^B(a_2-a_1)}.
\end{align}
Thus, we have
\begin{align*}
  \varphi_{A_2}^{\lambda_{a_1},\lambda_{a_2}}(x)&=\varphi_{A_2}(x)+\sum_{k=1}^2 \left\{\P_x\left(T_{a_k}=T_{A_2}\right)\cdot \lim_{q\to 0}r_q(0)\P_{a_k}\left[\Gamma_{\bm{e}_q}^{(2)}\right]\right\}\\
  &=\varphi_{A_2}(x)+\P_x\left(T_{a_1}<T_{a_2}\right)\cdot \lim_{q\to 0}r_q(0)\P_{a_1}\left[\Gamma_{\bm{e}_q}^{(2)}\right]\\
  &\qquad +\P_x\left(T_{a_2}<T_{a_1}\right)\cdot \lim_{q\to 0}r_q(0)\P_{a_2}\left[\Gamma_{\bm{e}_q}^{(2)}\right]\\
  &=\varphi_{A_2}(x)+\P_x\left(T_{a_1}<T_{a_2}\right)\cdot \frac{I_{1}+J_{1,2}I_{2}}{1-J_{1,2}J_{2,1}}\\
  &\qquad +\P_x\left(T_{a_2}<T_{a_1}\right)\cdot\frac{J_{2,1}I_{1}+I_{2}}{1-J_{1,2}J_{2,1}}\\
  &=\varphi_{A_2}(x)+\P_x\left(T_{a_1}<T_{a_2}\right)\cdot \frac{\frac{h(a_1-a_2)}{1+\lambda_{a_1}h^B(a_1-a_2)}+\frac{1}{1+\lambda_{a_1}h^B(a_1-a_2)}\frac{h(a_2-a_1)}{1+\lambda_{a_2}h^B(a_1-a_2)}}{1-\frac{1}{1+\lambda_{a_1}h^B(a_1-a_2)}\frac{1}{1+\lambda_{a_2}h^B(a_1-a_2)}}\\
  &\qquad +\P_x\left(T_{a_2}<T_{a_1}\right)\cdot\frac{\frac{1}{1+\lambda_{a_2}h^B(a_1-a_2)}\frac{h(a_1-a_2)}{1+\lambda_{a_1}h^B(a_1-a_2)}+\frac{h(a_2-a_1)}{1+\lambda_{a_2}h^B(a_1-a_2)}}{1-\frac{1}{1+\lambda_{a_1}h^B(a_1-a_2)}\frac{1}{1+\lambda_{a_2}h^B(a_1-a_2)}}\\
  &=\varphi_{A_2}(x)+\P_x(T_{a_1}<T_{a_2})\cdot \frac{1+\lambda_{a_2}h(a_1-a_2)}{\lambda_{a_1}+\lambda_{a_2}+\lambda_{a_1}\lambda_{a_2}h^B(a_1-a_2)}\\
  &\qquad +\P_x(T_{a_2}<T_{a_1})\cdot \frac{1+\lambda_{a_1}h(a_2-a_1)}{\lambda_{a_1}+\lambda_{a_2}+\lambda_{a_1}\lambda_{a_2}h^B(a_1-a_2)}.
  \stepcounter{equation}\tag{\theequation}
\end{align*}
This coincides with $\varphi_{a_1,a_2}^{(0),\lambda_{a_1},\lambda_{a_2}}(x)$ which is defined by Remark 1.4 of Iba-Yano \cite{IY-3}.

\subsection{One-point hitting time clock}
By (\ref{P3}), we have
\begin{align}
  I_1^{(\pm 1)}&=\frac{h^{(\pm 1)}(a_1-a_2)}{1+\lambda_{a_1}h^B(a_1-a_2)},\\
  I_2^{(\pm 1)}&=\frac{h^{(\pm 1)}(a_2-a_1)}{1+\lambda_{a_2}h^B(a_1-a_2)}.
\end{align}
Thus, we have
\begin{align*}
  \varphi_{A_2}^{(\pm 1),\lambda_{a_1},\lambda_{a_2}}(x)&=\varphi_{A_2}^{(\pm 1)}(x)+\sum_{k=1}^2 \left\{\P_x\left(T_{a_k}=T_{A_2}\right)\cdot \lim_{b\to \pm \infty}h^B(b)\P_{a_k}\left[\Gamma_{T_b}^{(2)}\right]\right\}\\
  &=\varphi_{A_2}^{(0)}(x)+\P_x\left(T_{a_1}<T_{a_2}\right)\cdot \frac{I_1^{(\pm 1)}+J_{1,2}I_2^{(\pm 1)}}{1-J_{1,2}J_{2,1}}\\
  &\qquad +\P_x\left(T_{a_2}<T_{a_1}\right)\cdot\frac{J_{2,1}I_1^{(\pm 1)}+I_2^{(\pm 1)}}{1-J_{1,2}J_{2,1}}\\
    &=\varphi_{A_2}^{(\pm 1)}(x)+\P_x\left(T_{a_1}<T_{a_2}\right)\cdot \frac{\frac{h^{(\pm 1)}(a_1-a_2)}{1+\lambda_{a_1}h^B(a_1-a_2)}+\frac{1}{1+\lambda_{a_1}h^B(a_1-a_2)}\frac{h^{(\pm 1)}(a_2-a_1)}{1+\lambda_{a_2}h^B(a_1-a_2)}}{1-\frac{1}{1+\lambda_{a_1}h^B(a_1-a_2)}\frac{1}{1+\lambda_{a_2}h^B(a_1-a_2)}}\\
  &\qquad +\P_x\left(T_{a_2}<T_{a_1}\right)\cdot\frac{\frac{1}{1+\lambda_{a_2}h^B(a_1-a_2)}\frac{h^{(\pm 1)}(a_1-a_2)}{1+\lambda_{a_1}h^B(a_1-a_2)}+\frac{h^{(\pm 1)}(a_2-a_1)}{1+\lambda_{a_2}h^B(a_1-a_2)}}{1-\frac{1}{1+\lambda_{a_1}h^B(a_1-a_2)}\frac{1}{1+\lambda_{a_2}h^B(a_1-a_2)}}\\
  &=\varphi_{A_2}^{(\pm 1)}(x)+\P_x(T_{a_1}<T_{a_2})\cdot \frac{1+\lambda_{a_2}h^{(\pm 1)}(a_1-a_2)}{\lambda_{a_1}+\lambda_{a_2}+\lambda_{a_1}\lambda_{a_2}h^B(a_1-a_2)}\\
  &\qquad +\P_x(T_{a_2}<T_{a_1})\cdot \frac{1+\lambda_{a_1}h^{(\pm 1)}(a_2-a_1)}{\lambda_{a_1}+\lambda_{a_2}+\lambda_{a_1}\lambda_{a_2}h^B(a_1-a_2)}.
  \stepcounter{equation}\tag{\theequation}
\end{align*}
This coincides with $\varphi_{a_1,a_2}^{(\pm 1),\lambda_{a_1},\lambda_{a_2}}(x)$ which is defined by Remark 1.4 of Iba-Yano \cite{IY-3}.

Similarly, it can be seen that the cases of the two-point hitting time clock and the inverse local time clock also coincide with that of Remark 1.4 of Iba-Yano \cite{IY-3}.


\section{Penalized measure}
\label{Sec8}
In this section, we study the penalized measure, that is a measure obtained as the limit in a penalization problem. By Theorem 1.7 of Takeda-Yano \cite{TY}, we can define the one-point penalized measure as
\begin{align}
  \Q_{x}^{(\gamma,1)}\Big|_{\F_t}= \frac{\varphi_{A_1}^{(\gamma),\lambda_{a_1}}(X_s)\Gamma_s^{(1)}}{\varphi_{A_1}^{(\gamma),\lambda_{a_1}}(X_0)\Gamma_0^{(1)}}\cdot \P_x\Big|_{\F_t},
\end{align}
where 
\begin{align}
  \varphi_{A_1}^{(\gamma),\lambda_{a_1}}(x):=h^{(\gamma)}(x-a_1)+\frac{1}{\lambda_{a_1}}.
\end{align}
Moreover, from the preceding discussion, for $n\ge 2$, we can define the $n$-point penalized measure as
\begin{align}
  \Q_{x}^{(\gamma,n)}\Big|_{\F_t}=\frac{\varphi_{A_n}^{(\gamma),\lambda_{a_1},...,\lambda_{a_n}}(X_s)\Gamma_s^{(n)}}{\varphi_{A_n}^{(\gamma),\lambda_{a_1},...,\lambda_{a_n}}(x)}\cdot \P_x\big|_{\F_t}.
\end{align}
Note that the measure $\Q_x^{(\gamma,n)}$ can be well-defined on $\F_\infty$ (see, e.g., Theorem 9.1 of \cite{Yano}).

First, we describe the behavior of the process under the penalized measure. Although the proof of the following proposition is parallel to Theorem 1.4 of Takeda \cite{Takeda}, we give the proof for completeness of this paper.\\

\begin{prop}
  For $n\ge 1$, the $n$-point penalized process $((X_t)_{t\ge 0},\Q_{x}^{(\gamma,n)})$ is transient.
\end{prop}
\begin{proof}
  For $0<s<t$ and a non-negative bounded $\F_s$-measurable functional $F_s$, we have
  \begin{align*}
    \Q_x^{(\gamma,n)}\left[\frac{1}{\varphi_{A_n}^{(\gamma),\lambda_{a_1},...,\lambda_{a_n}}(X_t)}\cdot F_s\right]&=\frac{1}{\varphi_{A_n}^{(\gamma),\lambda_{a_1},...,\lambda_{a_n}}(x)}\P_x\left[F_s\cdot \Gamma_t^{(n)}\right]\\
    &\le \frac{1}{\varphi_{A_n}^{(\gamma),\lambda_{a_1},...,\lambda_{a_n}}(x)}\P_x\left[F_s\cdot \Gamma_s^{(n)}\right]\\
    &=\Q_x^{(\gamma,n)}\left[\frac{1}{\varphi_{A_n}^{(\gamma),\lambda_{a_1},...,\lambda_{a_n}}(X_s)}\cdot F_s\right].
    \stepcounter{equation}\tag{\theequation}
  \end{align*}
  Thus, $(\frac{1}{\varphi_{A_n}^{(\gamma),\lambda_{a_1},...,\lambda_{a_n}}(X_t)})_{t\ge 0}$ is a non-negative $\Q_x^{(\gamma,n)}$-supermartingale. By the martingale convergence theorem, 
  \begin{align}
    \lim_{t\to \infty} \frac{1}{\varphi_{A_n}^{(\gamma),\lambda_{a_1},...,\lambda_{a_n}}(X_t)}\ \text{exists}\qquad \Q_x^{(\gamma,n)}\text{-a.s.}
  \end{align}
  By Fatou's lemma and recurrence of $((X_t)_{t\ge 0},\P_x)$, we have
  \begin{align*}
    \Q_x^{(\gamma,n)}\left[\lim_{t\to \infty}\frac{1}{\varphi_{A_n}^{(\gamma),\lambda_{a_1},...,\lambda_{a_n}}(X_t)}\right]&\le \varliminf_{t\to \infty}\Q_x^{(\gamma,n)}\left[\frac{1}{\varphi_{A_n}^{(\gamma),\lambda_{a_1},...,\lambda_{a_n}}(X_t)}\right]\\
    &=\varliminf_{t\to \infty}\frac{1}{\varphi_{A_n}^{(\gamma),\lambda_{a_1},...,\lambda_{a_n}}(x)}\P_x\left[\Gamma_t^{(n)}\right]\\
    &=0.
    \stepcounter{equation}\tag{\theequation}
  \end{align*}
  This implies
  \begin{align}
    \lim_{t\to \infty}\frac{1}{\varphi_{A_n}^{(\gamma),\lambda_{a_1},...,\lambda_{a_n}}(X_t)}=0\qquad \Q_x^{(\gamma,n)}\text{-a.s.}
  \end{align}
  Thus, by the definition of $\varphi_{A_n}^{(\gamma),\lambda_{a_1},...,\lambda_{a_n}}(x)$, we obtain
  \begin{align}
  \label{tran}
    \lim_{t\to \infty}|X_t|=\infty\qquad \Q_x^{(\gamma,n)}\text{-a.s.}
  \end{align}
  It implies that the process $((X_t)_{t\ge 0},\Q_x^{(\gamma,n)})$ is transient.
\end{proof}

Since the process $(X_t)_{t\ge 0}$ was assumed to be recurrent under the measure $\P_x$, this proposition shows that the measures $\P_x$ and $\Q_x^{(\gamma,n)}$ are mutually singular on $\F_\infty$. However on $\F_t$, the measures $\P_x$ and $\Q_x^{(\gamma,n)}$ are equivalent. 

Next, for $n\ge 1$, we define the \emph{unweighted measure} of $\Q_x^{(\gamma,n)}$ by
\begin{align}
\mathscr{Q}_{x}^{(\gamma,n)}&:=\frac{\varphi_{A_n}^{(\gamma)}(x)}{\Gamma_\infty^{(n)}}\cdot \Q_{x}^{(\gamma,n)}\qquad \text{on}\ \F_\infty.
\end{align}
The unweighted measures between $n$-point penalized measures actually coincide.\\

\begin{prop}
  For $n\ge 2$, it holds that
  \begin{align}
    \mathscr{Q}_{x}^{(\gamma,1)}= \mathscr{Q}_{x}^{(\gamma,n)}.
  \end{align}
\end{prop}
\begin{proof}
By (\ref{tran}), we have
\begin{align}
  \lim_{t\to \infty}\frac{\varphi_{A_n}^{(\gamma),\lambda_{a_1},...,\lambda_{a_n}}(X_t)}{h^{(\gamma)}(X_t-a_n)}=1\qquad Q_x^{(\gamma,n)}\text{-a.s.}
\end{align}
Therefore, we apply to Theorem 4.1 of Yano \cite{Yano} as $\Gamma_t=\Gamma_t^{(n)}$ and $\mathscr{E}_t=\Gamma_t^{(1)}=e^{-\lambda_{a_1}L_t^{a_1}}$, then the assertion holds.
\end{proof}

Thanks to this proposition, we have the explicit formula between $n$-point penalized measures.\\

\begin{cor}
  For $n\ge 2$, it holds that
  \begin{align}
    \Q_{x}^{(\gamma,n)}=\frac{\varphi_{A_1}^{(\gamma),\lambda_{a_1}}(x)}{\varphi_{A_n}^{(\gamma),\lambda_{a_1},...,\lambda_{a_n}}(x)}e^{-(\lambda_{a_2} L_\infty^{a_2}+\cdots +\lambda_{a_n}L_\infty^{a_n})}\cdot \Q_{x}^{(\gamma,1)}\qquad \text{on}\ \F_\infty.
  \end{align}
  Moreover, by considering the measure of the whole space, we obtain
  \begin{align}
    \varphi_{A_n}^{(\gamma),\lambda_{a_1},...,\lambda_{a_n}}(x)=\varphi_{A_1}^{(\gamma),\lambda_{a_1}}(x)\Q_{x}^{(\gamma,1)}\left[e^{-(\lambda_{a_2} L_\infty^{a_2}+\cdots +\lambda_{a_n}L_\infty^{a_n})}\right].
  \end{align}
\end{cor}

Finally, we consider the distribution of the final local time under $\Q_x^{(\gamma,n)}.$\\

\begin{prop}
  Let $c\in \R$. It holds that
  \begin{align*}
    \Q_{x}^{(\gamma,n)}(L_\infty^c\in dt)&=\frac{\varphi_{A_n}^{(\gamma),\lambda_{a_1},...,\lambda_{a_n}}(c)}{\varphi_{A_n}^{(\gamma),\lambda_{a_1},...,\lambda_{a_n}}(x)}\P_x\left[\Gamma_{T_c}^{(n)}\right]\bm{n}^c\left[1-\Gamma_{T_c}^{(n)}\right]e^{-t \bm{n}^c\left[1-\Gamma_{T_c}^{(n)}\right]}dt\\
&\qquad +\left(1-\frac{\varphi_{A_n}^{(\gamma),\lambda_{a_1},...,\lambda_{a_n}}(c)}{\varphi_{A_n}^{(\gamma),\lambda_{a_1},...,\lambda_{a_n}}(x)}\P_x\left[\Gamma_{T_c}^{(n)}\right]\right)\delta_0(dt).
\stepcounter{equation}\tag{\theequation}
  \end{align*}
  Consequently, by setting $x=c$, we obtain
  \begin{align}
    \Q_{c}^{(\gamma,n)}(L_\infty^{c}\in \cdot )\stackrel{d}{=} \mathrm{Exp}\left(\bm{n}^c\left[1-\Gamma_{T_c}^{(n)}\right]\right).
  \end{align}
\end{prop}
\begin{proof}
By the optional sampling theorem, we have
  \begin{align*}
  \Q_{x}^{(\gamma,n)}(L_s^c>t)&=\P_x\left[1_{\{L_s^c>t\}}\cdot \frac{\varphi_{A_n}^{(\gamma),\lambda_{a_1},...,\lambda_{a_n}}(X_s)\Gamma_s^{(n)}}{\varphi_{A_n}^{(\gamma),\lambda_{a_1},...,\lambda_{a_n}}(x)}\right]\\
  &=\P_x\left[1_{\{\eta_t^c<s\}}\cdot \frac{\varphi_{A_n}^{(\gamma),\lambda_{a_1},...,\lambda_{a_n}}(X_s)\Gamma_s^{(n)}}{\varphi_{A_n}^{(\gamma),\lambda_{a_1},...,\lambda_{a_n}}(x)}\right]\\
  &=\P_x\left[1_{\{\eta_t^c<s\}}\cdot \frac{\varphi_{A_n}^{(\gamma),\lambda_{a_1},...,\lambda_{a_n}}(X_{\eta_t^c})\Gamma_{\eta_t^c}^{(n)}}{\varphi_{A_n}^{(\gamma),\lambda_{a_1},...,\lambda_{a_n}}(x)}\right]\\
  &=\frac{\varphi_{A_n}^{(\gamma),\lambda_{a_1},...,\lambda_{a_n}}(c)}{\varphi_{A_n}^{(\gamma),\lambda_{a_1},...,\lambda_{a_n}}(x)}\P_x\left[1_{\{\eta_t^c<s\}}\Gamma_{\eta_t^c}^{(n)}\right].
  \stepcounter{equation}\tag{\theequation}
\end{align*}
Letting $s\to \infty$, we have by (\ref{P4}),
\begin{align*}
  \Q_{x}^{(\gamma,n)}(L_\infty^c>t)&=\frac{\varphi_{A_n}^{(\gamma),\lambda_{a_1},...,\lambda_{a_n}}(c)}{\varphi_{A_n}^{(\gamma),\lambda_{a_1},...,\lambda_{a_n}}(x)}\P_x\left[\Gamma_{\eta_t^c}^{(n)}\right]\\
  &=\frac{\varphi_{A_n}^{(\gamma),\lambda_{a_1},...,\lambda_{a_n}}(c)}{\varphi_{A_n}^{(\gamma),\lambda_{a_1},...,\lambda_{a_n}}(x)}\P_x\left[\Gamma_{T_c}^{(n)}\right]\P_c\left[\Gamma_{\eta_t^c}^{(n)}\right]\\
   &=\frac{\varphi_{A_n}^{(\gamma),\lambda_{a_1},...,\lambda_{a_n}}(c)}{\varphi_{A_n}^{(\gamma),\lambda_{a_1},...,\lambda_{a_n}}(x)}\P_x\left[\Gamma_{T_c}^{(n)}\right]e^{-t \bm{n}^c\left[1-\Gamma_{T_c}^{(n)}\right]}.
\end{align*}
In particular, we have
\begin{align*}
  \Q_{x}^{(\gamma,n)}(L_\infty^c=0)&=1-\Q_{x}^{(\gamma,n)}(L_\infty^c>0)\\
  &=1-\frac{\varphi_{A_n}^{(\gamma),\lambda_{a_1},...,\lambda_{a_n}}(c)}{\varphi_{A_n}^{(\gamma),\lambda_{a_1},...,\lambda_{a_n}}(x)}\P_x\left[\Gamma_{T_c}^{(n)}\right].
\end{align*}
Therefore, the proof is complete.
\end{proof}


\section{Appendix: the case of $\gamma=0$}
\label{S9}
In papers Takeda-Yano \cite{TY} and Iba-Yano \cite{IY-3}, the limits obtained using the exponential clock and those obtained using the two-point hitting time clock with $\gamma=0$ are consistent. Therefore, the same outcome is expected in the present case. To demonstrate that they are indeed consistent, it is sufficient to show the following identity:
\begin{align*}
\label{g0}
  &\frac{1}{n^{a_k}(T_{A_n\setminus \{a_k\}}<\infty)+\lambda_{a_k}}\left(1-\sum_{\substack{i;\ i\le n\\ i\neq k}}h(a_i-a_k)n^{a_k}(T_{a_i}=T_{A_n\setminus \{a_k\}}<\infty)\right)\\
  &\qquad =\frac{1}{\lambda_{a_k}}-\sum_{\substack{i;\ i\le n\\ i\neq k}}\P_{a_k}\left[e^{-\lambda_{a_k}L_{T_{a_i}}^{a_k}},\ T_{a_i}=T_{A_n\setminus \{a_k\}}\right]\left(\frac{1}{\lambda_{a_k}}+h(a_i-a_k)\right).
  \stepcounter{equation}\tag{\theequation}
\end{align*}

First, we show that the first terms coincide:\\
\begin{lem}
It holds that
\begin{align}
  \frac{1}{n^{a_k}(T_{A_n\setminus \{a_k\}}<\infty)+\lambda_{a_k}}=\frac{1}{\lambda_{a_k}}-\sum_{\substack{i;\ i\le n\\ i\neq k}}\P_{a_k}\left[e^{-\lambda_{a_k}L_{T_{a_i}}^{a_k}},\ T_{a_i}=T_{A_n\setminus \{a_k\}}\right]\frac{1}{\lambda_{a_k}}.
\end{align}
\end{lem}
\begin{proof}
In the same manner as in the proof of Lemma 6.3 of Takeda-Yano \cite{TY}, the following can be obtained:
\begin{align}
\label{Lexp}
  L_{T_{A_n\setminus \{a_k\}}}^{a_k}\ \text{under}\ \P_{a_k}\deq \text{Exp}\left(\bm{n}^{a_k}(T_{A_n\setminus \{a_k\}}<\infty)\right).
\end{align}
Thus, we have
\begin{align*}
&\frac{1}{\lambda_{a_k}}-\sum_{\substack{i;\ i\le n\\ i\neq k}}\P_{a_k}\left[e^{-\lambda_{a_k}L_{T_{a_i}}^{a_k}},\ T_{a_i}=T_{A_n\setminus \{a_k\}}\right]\frac{1}{\lambda_{a_k}}\\
&\qquad = \P_{a_k}\left[\frac{1-e^{-\lambda_{a_k}L_{T_{A_n\setminus \{a_k\}}}^{a_k}}}{\lambda_{a_k}}\right]\\
&\qquad =\int_0^\infty \left(\frac{1-e^{-\lambda_{a_k}x}}{\lambda_{a_k}}\right)\bm{n}^{a_k}(T_{A_n\setminus \{a_k\}}<\infty)e^{-\bm{n}^{a_k}(T_{A_n\setminus \{a_k\}}<\infty)x}dx\\
&\qquad =\frac{1}{\bm{n}^{a_k}(T_{A_n\setminus \{a_k\}}<\infty)+\lambda_{a_k}}.
    \stepcounter{equation}\tag{\theequation}
\end{align*}
The proof is complete.
\end{proof}

By this lemma, it suffices to show the following:\\
\begin{lem}
It holds that
\begin{align}
\label{I}
  \P_{a_k}\left[e^{-\lambda_{a_k}L_{T_{a_i}}^{a_k}},\ T_{a_i}=T_{A_n\setminus \{a_k\}}\right]=\frac{\bm{n}^{a_k}(T_{a_i}=T_{A_n\setminus \{a_k\}}<\infty )}{\bm{n}^{a_k}(T_{A_n\setminus \{a_k\}}<\infty)+\lambda_{a_k}}
\end{align}
for $i\le n$ and $i\neq k$.
\end{lem}
\begin{proof}
First, we consider the case of $\lambda_{a_k}=0$. Let $A:=\{T_{A_n\setminus \{a_k\}}<\infty\}.$ Since $\epsilon_{\sigma_A}^{a_k}$ has a distribution $\bm{n}^{a_k}(\cdot |A)$, we have
  \begin{align*}
  \label{ak0}
 \P_{a_k}\left(T_{a_i}=T_{A_n\setminus \{a_k\}}\right)&=\P_{a_k}\left(\epsilon_{\sigma_A}^{a_k}\in \{T_{a_i}=T_{A_n\setminus \{a_k\}}<\infty\}\right)\\
 &=\frac{\bm{n}^{a_k}(T_{a_i}=T_{A_n\setminus \{a_k\}}<\infty)}{\bm{n}^{a_k}(T_{A_n\setminus \{a_k\}}<\infty)}.
     \stepcounter{equation}\tag{\theequation}
  \end{align*}

Next, we consider the general case. Since
\begin{align*}
\P_{a_k}(L_{T_{a_i}}^{a_k}>l|\ T_{a_i}=T_{A_n\setminus \{a_k\}})&=\frac{\P_{a_k}(\eta_{l}^{a_k}<T_{a_k},\ T_{a_i}=T_{A_n\setminus \{a_k\}})}{\P_{a_k}(T_{a_i}=T_{A_n\setminus \{a_k\}})}\\
  &=\frac{\P_{a_k}(\eta_l^{a_k}<T_{A_n\setminus \{a_k\}})\P_{a_k}(T_{a_i}=T_{A_n\setminus \{a_k\}})}{\P_{a_k}(T_{a_i}=T_{A_n\setminus \{a_k\}})}\\
  &=\P_{a_k}\left(\eta_l^{a_k}<T_{A_n\setminus \{a_k\}}\right)\\
  &=\P_{a_k}\left(L_{T_{A_n\setminus \{a_k\}}}^{a_k}>l\right),
   \stepcounter{equation}\tag{\theequation}
\end{align*}
we have by (\ref{Lexp}),
\begin{align}
   L_{T_{a_i}}^{a_k}\ \text{under}\ \P_{a_k}(\cdot|\ T_{a_i}=T_{A_n\setminus \{a_k\}})\deq \text{Exp}\left(\bm{n}^{a_k}(T_{A_n\setminus \{a_k\}}<\infty)\right).
\end{align}
Thus, by (\ref{nMP}) and (\ref{ak0}), we have
\begin{align*}
&\P_{a_k}\left[e^{-\lambda_{a_k}L_{T_{a_i}}^{a_k}},\ T_{a_i}=T_{A_n\setminus \{a_k\}}\right]\\
  &\qquad =\P_{a_k}(T_{a_i}=T_{A_n\setminus \{a_k\}})\P_{a_k}\left[e^{-\lambda_{a_k}L_{T_{a_i}}^{a_k}}\Big|\ T_{a_i}=T_{A_n\setminus \{a_k\}}\right]\\
  &\qquad =\P_{a_k}(T_{a_i}=T_{A_n\setminus \{a_k\}})\int_0^\infty e^{-\lambda_{a_k}x}\bm{n}^{a_k}(T_{A_n\setminus \{a_k\}}<\infty)e^{-\bm{n}^{a_k}(T_{A_n\setminus \{a_k\}}<\infty)x}dx\\
  &\qquad =\frac{\P_{a_k}(T_{a_i}=T_{A_n\setminus \{a_k\}})\bm{n}^{a_k}(T_{A_n\setminus \{a_k\}}<\infty)}{\bm{n}^{a_k}(T_{A_n\setminus \{a_k\}}<\infty)+\lambda_{a_k}}\\
  &\qquad =\frac{\bm{n}^{a_k}(T_{a_i}=T_{A_n\setminus \{a_k\}}<\infty)}{\bm{n}^{a_k}(T_{A_n\setminus \{a_k\}}<\infty)+\lambda_{a_k}}.
   \stepcounter{equation}\tag{\theequation}
\end{align*}
The proof is complete.
\end{proof}

Therefore, we obtain the identity (\ref{g0}).


\section*{Acknowledgement}
The author would like to thank Professor Kouji Yano, for his careful guidance and great support. This work was supported by JST SPRING, Grant Number JPMJSP2138. 

\bibliographystyle{plain}

\end{document}